\documentclass[noccpublish]{cc}

\usepackage[german, english]{babel}
\usepackage{pst-all}
\usepackage[labelsep=none]{caption}
\newcommand{\prip}{Q}
\newcommand{\priP}{S}
\newcommand{\priq}{R}
\newcommand{\et}{\tilde{e}}
\newcommand{\lcm}{\mbox{lcm}}
\newcommand{\Pinf}{P_\infty}
\newcommand{\priPinf}{\priP_\infty}
\newcommand{\pripinf}{\prip_\infty}
\newcommand{\priqinf}{\priq_\infty}
\newcommand{\ff}{F}
\newcommand{\fmt}{\varphi}

\title{Decomposition of Polynomials}

\author{Raoul Blankertz\\
    b-it cosec, University of Bonn\\
    Dahlmannstr. 2\\
    D-53113 Bonn \\
    Germany\\
    \url{blankertz@bit.uni-bonn.de}
    }

\begin{abstract}
This diploma thesis is concerned with functional decomposition $f = g \circ h$ of polynomials. First an algorithm is described which computes decompositions in polynomial time. This algorithm was originally proposed by Zippel (1991). A bound for the number of minimal collisions is derived. Finally a proof of a conjecture in von zur Gathen, Giesbrecht \& Ziegler (2010) is given, which states a classification for a special class of decomposable polynomials.
\end{abstract}

\begin{document}

\begin{note*}
This is a modified version of the author's diploma thesis. The main changes concern notation and rephrasing of some results.
\end{note*}

\newpage
\thispagestyle{empty}
\mbox{}
\newpage
\thispagestyle{empty}
\section*{Danksagung}

Ich m\"ochte mich an dieser Stelle bei den Personen bedanken, die mir bei der Erarbeitung dieser Diplomarbeit geholfen haben. 

Mein Dank gilt Herrn Prof.~Dr.~Joachim von zur Gathen f\"ur die Betreuung dieser Diplomarbeit. Besonders zu sch\"atzen wei\ss\ ich seine kritischen Anmerkungen und die anregenden Gespr\"ache mit ihm.
Au\ss erdem bedanke ich mich bei Herrn Konstantin Ziegler f\"ur seine Betreuung und f\"ur die  
zahl\-reichen fachlichen Diskussionen, Ratschl\"age und Hilfestellungen.
Bei Herrn Prof.~Dr.~Nitin Saxena bedanke ich mich f\"ur die \"Ubernahme der Zweitkorrektur.

F\"ur die sprachliche Durchsicht bedanke ich mich bei David Steimle und Jens Humrich. 

Nicht zuletzt m\"ochte ich mich bei meiner Familie bedanken. Bei meinen Eltern und Gro\ss eltern bedanke ich mich f\"ur die Unterst\"utzung w\"ahrend des Studiums. Ganz besonders bedanke ich mich bei meiner Freundin Ellen f\"ur ihre Geduld und ihren Beistand. 

\newpage
\thispagestyle{empty}
\mbox{}
\newpage
\thispagestyle{empty}
\tableofcontents
\newpage
\thispagestyle{empty}
\mbox{}
\newpage

\setcounter{page}{1}
\section{Introduction}
\label{sec:intro}

\subsection{Deutsche Einleitung}
Die Komposition zweier Polynome ist ebenfalls ein Polynom. Umgekehrt kann man sich fragen, unter welchen Umst\"anden ein Polynom die Komposition zweier anderer Polynome ist -- oder in anderen Worten: Wann ist ein Polynom funktional zerlegbar? \cite{rit22} besch\"aftigte sich mit dieser Frage, wobei er als Grundk\"orper $\mathbb{C}$ voraussetzte.
Im Gegensatz zur multiplikativen Zerlegung ist die funktionale Zerlegung nicht eindeutig. Ritts erstes Theorem besagt, dass die Grade der Komponenten zweier verschiedener vollst\"andigen Zerlegungen (das hei\ss t die Komponenten sind unzerlegbar) modulo einer Permutation gleich sind. Ritts zweites Theorem gibt eine Klassifikation der L\"osungen von $g_1 \circ h_1 = g_2 \circ h_2$ mit $\deg(g_1) = \deg(h_2)$ an. Diese beiden Theoreme konnten auf K\"orper der Charakteristik Null verallgemei\-nert werden (siehe \cite{dorwha74}). Aller\-dings gibt es Gegenbeispiele f\"ur 
beide Theoreme \"uber K\"orper mit positiver Charakteristik. In solchen K\"orpern treten so genannte gleichgradige Kollisionen auf. Diese Kollisionen machen es schwer, die Anzahl der zerlegbaren Polynome \"uber einem endlichen K\"orper anzugeben oder auch nur zu sch\"atzen, siehe \cite{gat09b}. Auch algorithmisch ist das Problem im Falle positiver Charakteristik schwieriger, siehe \cite{gat90c} und \cite{gat90d}. 

Zuerst werden in Kapitel \bare\ref{sec:dec} grundlegende Konzepte eingef\"uhrt. In Kapitel \bare\ref{sec:algo_minDec} wird dann ein Algorithmus von \cite{zip91} besprochen. Der von Zippel vorgestellte Algorithmus berechnet Zerlegungen von rationalen Funktionen in polynomieller Zeit. Dabei zitiert Zippel einige Resultate aus \cite{lanmil85}, auf denen sein Algorithmus basiert. Diese Resultate wurden aller\-dings dort nur f\"ur den Grundk\"orper $\mathbb{Q}$ bewiesen und nicht wie ben\"otigt \"uber einem Funktionenk\"orper $\ff (t)$ von beliebiger Charakteristik. Eine komplette Beschreibung dieses Algorithmus f\"ur die Zerlegung von Polynomen und ein Beweis seiner Korrektheit wird in dieser Arbeit gegeben. Dabei wird zun\"achst eine Beziehung zwischen der funktionalen Zerlegung eines Polynoms und den Bl\"ocken einer bestimmten Permutationsgruppe hergestellt. Dann wird gezeigt wie minimale Bl\"ocke effizient berechnet werden k\"onnen und wie man daraus die entsprechende Zerlegung gewinnt. Am Ende von Kapitel \bare\ref{sec:algo_minDec} wird eine obere Schranke f\"ur die Anzahl von minimalen Zerlegungen eines Polynoms hergeleitet.

Danach widmet sich Kapitel \bare\ref{sec:conj} der Klassifikation von Polynomen vom Grad $p^2$ mit mindestens zwei verschiedenen Zerlegungen \"uber einem K\"orper der Charakteristik $p$. Diese Klassifikation wurde in \cite*{gatgie10} vorgeschlagen und soll nun in dieser Arbeit bewiesen werden. Der Beweis orientiert sich an den Beweisen von Ritts zweitem Theorem in den Arbeiten von \cite{dorwha74} und \cite{zan93}. 

\subsection{English introduction}
The functional composition of two polynomials is a polynomial itself. Conversely one could ask, when is a given polynomial the composition of two others polynomials -- or in other words: When is a polynomial functionally decomposable? 

At first in \ref{sec:dec} basic notions and concepts will be introduced. An algorithm for computing decompositions, which was originally proposed in \cite{zip91} \footnote{There is a related (unpublished) paper by Zippel from 1996, which was not known to the author until the submission of the thesis. A subsequent publication of the author's work will refer to it.}, will be discussed in \ref{sec:algo_minDec}. 
The proof relies on a generalization of results of \cite{lanmil85}. But this generalization lacked a proper foundation. A proof of correctness and a runtime estimation is provided (20 years later) in this paper.
In the end of \ref{sec:algo_minDec} an upper bound for the number of minimal decompositions of a polynomial will be deduced. 

In \cite*{gatgie10} a classification for decomposable polynomials of degree $p^2$ over a field of characteristic $p$ was proposed. This conjecture will be stated and proven in \ref{sec:conj}.

\section{Decompositions}
\label{sec:dec}
Let $\ff$ be an arbitrary field. In the runtime considerations of the algorithm in \ref{sec:algo_minDec} we restrict $\ff$ to a field in which one can compute efficiently and in \ref{sec:conj} we restrict $\ff$ to a field of positive characteristic. One can think of $\ff$ being a finite field, which is the most interesting case.

\begin{definition}
A polynomial $f$ in $\ff[x]$ is \emph{decomposable} if there are $g$ and $h$ in $\ff[x]$, both of degrees at least two, such that $f = g \circ h$. The pair $(g,h)$ is called a \emph{decomposition} of $f$. A polynomial is \emph{indecomposable} if it is not decomposable.

We call a polynomial \emph{original} if its graph passes though the origin, or, equivalently, it has a root at zero.
A polynomial is \emph{normal} if it is monic and original. We call a decomposition $(g,h)$ \emph{normal} if $g$ and $h$ are normal and we call it \emph{minimal} if $h$ is normal and indecomposable.  
\end{definition}

In a decomposition $(g, h)$, $g$ is uniquely determined by $f$ and $h$, since the ring homomorphism $\ff[x] \rightarrow \ff[x]$ with $ g \mapsto g \circ h$ is injective. Furthermore, $g$ is easy to compute by the generalized Taylor expansion, see \cite{gatger99}.

\begin{definition}
\label{def:lcomp} 
A linear left composition (linear right composition, linear composition) of $f$ is the polynomial $\ell \circ f$ ($f \circ \ell$, $\ell \circ f \circ \hat{\ell}$, respectively) for some linear polynomials $\ell$ and $\hat{\ell}$. 

The \emph{conjugate} of a normal polynomial $f$ by a linear polynomial $x + w$ is the normal polynomial $(x - f(w) ) \circ f \circ (x + w)$.

\end{definition}

Each non-constant polynomial has a unique linear left composition which is normal. Namely, if $a$ is the leading coefficient of $f$ and $b$ is its constant term, then $(a^{-1} x -a^{-1}b) \circ f$ is normal. The functional inverse of a linear polynomial $\ell=ax +b$ is $\ell^{-1} = a^{-1}x - a^{-1}b$. If a normal polynomial $f$ has a decomposition $(g, h)$ and $l$ is a linear polynomial such that $\ell \circ h$ is normal, then $g \circ \ell^{-1}$ is normal. This is because the leading coefficient and constant term of $(g \circ \ell^{-1}) \circ (\ell \circ h) = f$ equal the leading coefficient and the constant term, respectively, of $g \circ \ell^{-1}$. 

Functional decomposition is related to intermediate fields of certain field extensions in the following way. Let $\ff (t)$ be the function field in $t$ over $\ff$. Then for a given non-constant polynomial $f \in \ff[x]$ let $\fmt$ be the polynomial $f - t$ in $\ff (t)[x]$. Then $\fmt$ is irreducible by the Eisenstein criterion.

If we assume that the derivative $f'$ of $f$ is not zero then the derivative of $\fmt$ with respect to $x$ is not zero and thus $\fmt$ is separable. In this case, for a root $\alpha$ of $\fmt$, $\ff(t)[\alpha] = \ff(\alpha)$ is a separable field extension of $\ff(t)$.

In characteristic $0$ we have $f' \neq 0$.
If the characteristic of $\ff$ is $p$ and $f' = 0$ then there exists $\tilde{f}$ such that $f = \tilde{f}(x^{p^r} )$ and $\tilde{f}' \neq 0 $. If $\ff$ is finite or is the algebraic closure of a finite field then the Frobenius endomorphism $x \mapsto x^p$ is an automorphism of $\ff$. In this case, by knowing all decompositions of $\tilde{f}$ one knows all decompositions of $f$. 
In general the Frobenius endomorphism is not an automorphism (for example on function fields), but we will anyway assume that $f' \neq 0$. This assumption excludes some cases in general, but we lose no generality if $\ff$ is a finite field.

Now the following theorem states a correspondence between decompositions of $f$ and intermediate fields of $\ff(\alpha) \mid \ff(t)$. A proof of it can be found in  \cite{frimac69}.

\begin{theorem}
\label{thm:bij}
Let $f$ be a polynomial over $\ff$ with $f' \neq 0$ and let $\alpha$ be a root of $f - t \in \ff(t)[x]$. 
Let $L = \{ h \in \ff[x] \colon h \mbox{ is normal and } \exists g \in \ff[x] \colon f = g \circ h \}$ and let $M$ be the set of intermediate fields between $\ff(\alpha)$ and $\ff(t)$. Then the map $L \rightarrow M$ with $h \mapsto \ff(h(\alpha))$ is bijective.
\end{theorem}

The minimal polynomial of $\alpha$ over $\ff(h(\alpha))$ is $h(x) - h(\alpha)$. Thus we have $[\ff(\alpha) \colon \ff(h(\alpha)) ] = \deg (h)$.

The set $M$ is a lattice with the inclusion as order (the intersection and the composition of fields are the meet and the join). 
It is clear that if $h = g \circ h^*$, then $\ff(h(\alpha)) \subseteq \ff(h^*(\alpha))$. Thus, if we take $h^* \leq h $ to mean that $h = g \circ h^*$ for some $g \in \ff[x]$, then the bijection in \ref{thm:bij} is an order-reversing bijection of partially ordered sets. Thus $(L, \leq)$ is a lattice, which we call the \emph{lattice of decompositions} of $f$.

\begin{lemma}
Let $f \in \ff[x]$ with $f' \neq 0$ and $\ell_1$, $\ell_2 \in \ff[x]$ be linear. Let $\alpha$ be a root of $f  -t$ and $\beta$ a root of $\ell_1 \circ f \circ \ell_2 -t$. Then the field extensions $\ff(\alpha) \mid \ff (t)$ and $\ff(\beta) \mid \ff(t)$ are isomorphic.
\end{lemma}
\begin{proof}
Let $\Phi \colon \ff[\alpha] \rightarrow \ff[\beta]$ with $\alpha \mapsto \ell_2(\beta)$ be the evaluation homomorphism. Since $\beta$ is transcendental over $\ff$ this homomorphism is injective and extends to a field homomorphism $\Phi \colon \ff(\alpha) \rightarrow \ff(\beta)$. From $\Phi(t) = \Phi(f(\alpha)) = f (\ell_2(\beta)) = \ell_1^{-1} (t)$ follows that $\ff(t)$ is mapped to $\ff(t)$ under $\Phi$. The degrees of the extensions are equal, hence $\Phi$ is surjective.
\end{proof}

\begin{corollary}
If $\hat{f}$ is a linear composition of $f$, then the lattice of decompositions of $\hat{f}$ is isomorphic to the lattice of decompositions of $f$.
\end{corollary}
\begin{proof}
Let $\alpha$ and $\beta$ be roots of $f - t$ and $\hat{f} - t$, respectively. Then the lattice of decompositions of $\hat{f}$ is isomorphic to the lattice of intermediate fields between $\ff(\beta)$ and $\ff(t)$. By the previous lemma, this lattice is isomorphic to the lattice of intermediate fields between $\ff(\alpha)$ and $\ff(t)$, which is in turn isomorphic to the lattice of decompositions of $f$.
\end{proof}

Thus, one needs only to consider normal polynomials and normal decompositions. Furthermore the lattice of decompositions of a normal polynomial is invariant under conjugation.

\section{Finding minimal decompositions}
\label{sec:algo_minDec}

An algorithm that computes functional decompositions of rational functions was proposed in \cite{zip91}. In this paper Zippel cites results of \cite{lanmil85}, on which the algorithm relies. But these results were only proven for the ground field $\mathbb{Q}$ -- instead of $\ff(t)$, which would be needed. 
In the 20 years since then, nobody seems to have undertaken the somewhat ungrateful task of verifying whether Zippel's claims are actually true. 
A complete description of the algorithm for polynomial decomposition and a proof of its correctness will be given in this section.

The main idea for the algorithm is to relate decompositions of $f$ to certain partitions of the set of roots of $\fmt = f - t$ and to find an efficient way to compute these partitions. To specify this idea, let $(g,h)$ be a decomposition of $f$. For each root $\lambda$ of $g - t$ the roots of $h - \lambda$ form a subset of the roots of $\fmt$. Furthermore two different roots of $g- t$ yield two disjoint subsets. In this way one can partition the set of roots of $\fmt$ with respect to a decomposition of $f$. For getting a better understanding of the nature of these partitions we consider the notion of blocks.

\subsection{Blocks of imprimitivity}
\label{subsec:blocks} 
We introduce the notion of blocks of imprimitivity and its relation to decompositions. 
For this propose consider a finite permutation group $G$ on a finite set $Z$ (that is a subgroup $G \subseteq S(Z)$, where $S(Z)$ is the symmetric group on $Z$).
The following facts are mainly taken from \cite{wie64}. 

\begin{definition}
A subset $B$ of $Z$ is a \emph{block} of $G$ if for all $\sigma$ in $G$, the set $\sigma (B) \cap B$ is empty or equals $B$. 
\end{definition}

Equivalently, $B$ is a block of $G$ if for all $\sigma$ in $G$ the sets $B$ and  $\sigma(B)$ are disjoint or equal. If $B$ is a block, then any $\sigma (B)$ is a block. If $G$ is transitive and $B \neq \emptyset$ then $\{ \sigma (B) \}_{\sigma \in G}$ is a partition of $Z$ and is called a \emph{complete block system}. 

\begin{definition}
For a subgroup $U \subseteq G$ and $\alpha \in Z$ the \emph{orbit} of $\alpha$ under $U$ is the subset $U(\alpha) = \{ \sigma (\alpha) \mid \sigma \in U\}$. 
For a subset $S \subseteq Z$ the \emph{stabilizer} of $S$ is the subgroup $G_S = \{ \sigma \mid \sigma (S) = S \}$. We write $G_\alpha$ for $G_{\{\alpha\}}$.
A permutation group $G$ on $Z$ is called \emph{regular} if $G_\alpha$ is trivial for all $\alpha$ in $Z$.
\end{definition}

For $\sigma \in G$ we have $\sigma G_\alpha \sigma^{-1} = G_{\sigma (\alpha)}$. 
In particular, if $G$ is transitive, all stabilizers have the same cardinality, and $G$ is regular if and only if $G_\alpha$ is trivial for some $\alpha \in Z$.

\begin{lemma}
\label{lem:inters}
If $B$ and $C$ are blocks then $B \cap C$ is a block.
\end{lemma}
\begin{proof}
Let $\sigma$ be in $G$. Then $\sigma (B \cap C) \cap (B \cap C) = ( \sigma B \cap B ) \cap (\sigma C \cap C)$ and this is empty if and only if $\sigma B \cap B$ or $\sigma C \cap C$ is empty. If both are nonempty we get $\sigma (B \cap C) \cap (B \cap C) = B \cap C$, since $B$ and $C$ are blocks.
\end{proof}

\begin{definition}
The blocks $\emptyset$, $Z$, and $\{ \gamma \}$, for $\gamma \in Z$, are called \emph{trivial blocks}. A nontrivial block is called \emph{block of imprimitivity}. A permutation group $G$ on $Z$ is called \emph{primitive} if there are only trivial blocks. It is called \emph{imprimitive} otherwise. 
\end{definition}

\begin{example}
The alternating group $\mathbb{A}_n$ on $\{1, \ldots , n\}$ is primitive: Let without lose of generality $n > 2$ and assume $B$ is a block with at least two distinct elements, say $\alpha \neq \beta \in B$. Let $\gamma$ be an arbitrary element in $\{1, \ldots , n\}$ distinct from $\alpha$ and $\beta$. Then for $\sigma = (\alpha \ \beta \ \gamma) \in \mathbb{A}_n$ we have $\sigma(\alpha) = \beta \in \sigma B \cap B$. Therefore $\sigma B = B$, and since $\sigma (\beta) = \gamma \in \sigma B$ we get $\gamma \in B$. This proves $ B = \{1, \ldots , n\}$.

The same holds for $\mathbb{S}_n$.
\end{example}

\begin{example}
Let the dihedral group $D_6 = \langle \sigma , \tau \rangle $ act on $\{1, \ldots , 6\}$ by $\sigma = (1\ 3\ 5) (2\ 4\ 6)$ and $\tau = (1\ 4) (2\ 3)(5\ 6)$. Then $D_6$ is imprimitive, for example, $\{1,3,5\}$ and $\{1,2\}$ are nontrivial blocks.
\end{example}

The following theorem is essential for the link between the decomposition of polynomials and the theory of blocks. 
\begin{theorem}
\label{thm:wiel}
Let $G$ be a finite transitive permutation group on a finite set $Z$ and let $\alpha \in Z$. Then the lattice of subgroups between $G_\alpha$ and $G$ is isomorphic to the lattice of blocks containing $\alpha$.
\end{theorem}
\begin{proof}
For a block $B$ with $\alpha \in B$ define $\Phi \colon B \mapsto G_B$ and for a subgroup $G_\alpha \subseteq U \subseteq G$ define $\Psi \colon U \mapsto U(\alpha)$. To prove that both maps are well defined we first show that $G_B$ contains $G_\alpha$. Let $\sigma$ be in $G_\alpha$. Then $\sigma (\alpha) = \alpha \in B$. Thus $\sigma (B) \cap B \neq \emptyset$ and therefore $B = \sigma(B)$.

To prove that $U(\alpha)$ is a block, let $\sigma \in G$ and assume $\gamma$ is in $\sigma (U(\alpha)) \cap U(\alpha)$. Then there are $\tau$ and $\tau'$ in $U$ such that $\sigma \tau' (\alpha) = \gamma = \tau(\alpha)$. Thus $\tau^{-1} \sigma \tau' (\alpha) = \alpha$ and therefore $\tau^{-1} \sigma \tau' \in G_\alpha \subseteq U$. This implies that $\sigma$ is in $U$ and we have $\sigma (U(\alpha)) = U(\alpha)$.

Clearly $\Phi \circ \Psi (U) =\{ \sigma \in G \mid \sigma (U(\alpha)) =U(\alpha) \} \supseteq U$. For the reverse inclusion consider $\sigma \in G$ such that $\sigma(U (\alpha)) = U(\alpha)$. Thus there is $\tau \in U$ such that $\sigma \tau (\alpha) = \alpha$. Then $\sigma\tau \in G_\alpha \subseteq U$ and therefore $\sigma \in U$. Thus we have proven that $\Phi\circ\Psi = id$. For the other direction, one finds that $\Psi \circ\Phi (B) = \{ \sigma (\alpha) \mid \sigma(B) = B\} \subseteq B$. Let $\beta \in B$. Since $G$ acts transitively there is $\sigma \in G$ such that $\sigma (\alpha) = \beta$. Then $\beta \in B \cap \sigma B$ and thus $\sigma (B) = B$. Therefore $\beta = \sigma(\alpha) \in \Psi \circ \Phi (B)$. Thus $ \Psi \circ \Phi = id$ and we have proven that $ \Psi$ and $\Phi$ are bijective. It is now sufficient to show that $\Phi$ is order preserving. Let $B \subseteq B'$ and $\sigma \in G_B$. Then $B = \sigma (B) \cap B \subseteq \sigma (B') \cap B'$ and thus $\sigma (B') \cap B'$ is nonempty. Hence $\sigma$ is in $G_{B'}$. 
\end{proof}

We fix the following notation. Let $f$ be a polynomial in $ \ff [x]$ of degree $n$ with $f' \neq 0$. As before define $\fmt = f - t \in \ff(t)[x]$ and let $\alpha$ be a root of $\fmt$. Furthermore let $L$ be the splitting field of $\fmt$ over $\ff(t)$ and let $G$ be its Galois group. Then $G$ acts transitively on the set $Z$ of roots of $\fmt$. We consider $G$ as permutation group on $Z$.

\begin{corollary}
\begin{enumerate}
\item The lattice of decompositions of $f$ and the lattice of blocks of $G$ containing $\alpha$ are isomorphic.
\item Let $h$ be the right component of a normal decomposition of $f$ and let $B$ be the block corresponding to $h$. Then $\deg(h) = | B |$.
\end{enumerate}
\end{corollary}
\begin{proof}
The lattice of decompositions of $f$ is isomorphic to the lattice of intermediate fields of $\ff(\alpha) \mid \ff(t)$. This in turn is by Galois theory isomorphic to the lattice of subgroups between $G_\alpha$ and $G$. Thus, by the previous theorem one achieves an isomorphism between the lattice of decompositions of $f$ and the lattice of blocks containing $\alpha$.

Let $U$ be the subgroup corresponding to $h$ and $B$ be the corresponding block (that is, $\ff(h(\alpha)) = L^U$ and $U(\alpha) = B$). Then $\deg (h) = [\ff(\alpha) \colon \ff(h(\alpha)) ] =  [L^{G_\alpha} \colon L^U ] = ( U \colon G_\alpha)$. On the other hand, we have $ B= U(\alpha)  = \{ \sigma (\alpha) \mid \sigma \in U / G_\alpha \}$. Thus $|B| = (U \colon G_\alpha) = \deg(h)$.
\end{proof}

\subsection{Finding minimal blocks}
\label{subsec:findmin}
In this and in the next section we will discuss an algorithm that computes minimal blocks of the Galois group $G$. This algorithm and all intermediate results were introduced in \cite{lanmil85} for the ground field $\mathbb{Q}$. In our case we have the ground field $\ff (t)$, but the proofs are essentially based on \cite{lanmil85}.

From now on we consider only blocks containing $\alpha$. We call a nontrivial block $B$ minimal if all blocks $B' \subseteq B$ are either trivial or equal to $B$. If $B$ is minimal, then the corresponding decomposition is minimal.
\begin{lemma}
The set $B_\alpha = \{ \beta \in Z \mid \forall \sigma \in G_\alpha \colon \sigma \beta = \beta \}$ is a block of $G$.
\end{lemma}
\begin{proof}
Let $\beta \in B_\alpha$. Then each $\sigma \in G_\alpha$ fixes $\beta$, hence $G_\alpha \subseteq G_\beta$. Since $G$ is transitive we have  $|G_\alpha |=|G_\beta |$ and thus $G_\alpha = G_\beta$. 

Now let $\tau$ be in $G$ and assume $\tau B_\alpha \cap B_\alpha$ is not empty. Then there are $\beta$, $\beta' \in B_\alpha$ such that $\tau(\beta) = \beta'$. We have $G_\alpha = G_\beta = G_{\beta'}$. Let $\gamma \in B_\alpha$. Then $\tau^{-1} G_\alpha \tau = \tau^{-1} G_{\beta'} \tau = G_\beta = G_\gamma$. Thus for all $\sigma \in G_\alpha$ we have $\sigma \tau (\gamma) = \tau (\gamma)$. Hence $\tau( \gamma) \in B_\alpha$.
\end{proof}

Now factor $\fmt$ over $\ff (\alpha)$ into irreducible factors $\psi_i$ such that 
\begin{equation}\label{equ:factor} 
\fmt = \prod_ {i = 1 } ^s ( x - \alpha_i) \cdot \psi_{s+1} \cdot \ldots \cdot \psi_r,
\end{equation}
with $\alpha = \alpha_1$, $\alpha_i \in \ff(\alpha)$, and $\psi_i = x - \alpha_i$ for $1\leq i \leq s$, and $\deg \psi_i \geq 2$ for $s < i \leq r$. Since $\alpha_i \in \ff(\alpha)$ for $1\leq i \leq s$ there are rational functions $\ell_i$ such that $\alpha_i = \ell_i(\alpha_1)$. Since $\alpha$ is transcendental over $\ff$, from the equation $f(\alpha) = t = f(\ell_i(\alpha))$ follows that $\ell_i$ must be a linear polynomial. 
Clearly $\alpha_i $ is in $ B_\alpha$ for all $1 \leq i \leq s$. Let $\beta \in B_\alpha$. Then $\beta \in L^{G_\alpha} = \ff(\alpha)$. Thus $\beta = \ell(\alpha)$ for some linear polynomial $\ell$. We have proven that  $B_\alpha = \{ \alpha_i \mid 1 \leq i\leq s\}$. 
\begin{claim}
$H =(\{\ell_i \mid 1 \leq i \leq s\}, \circ )$ is a group.
\end{claim}
The neutral element in $H$ is $\ell_1 = x$. Since $\alpha$ is transcendental over $\ff$, the equation $f \circ \ell_i (\alpha) = t = f(\alpha)$ implies $f \circ \ell_i = f$. Then from $f ( \ell_i \circ \ell_j (\alpha)) = f ( \ell_j (\alpha)) = t$ follows that $ \ell_i \circ \ell_j (\alpha)$ is a root of $\fmt$ in $\ff(\alpha)$. Thus there exists $k$ such that $ \ell_i \circ \ell_j  = \ell_k$. In the same way by $f ( \ell_i ^{-1} (\alpha)) = f \circ \ell_i \circ \ell_i^{-1} (\alpha)= t$ we conclude the existence of the inverse of $\ell_i$ in $H$.

The following lemma presents us with the opportunity to lay hands on the Galois group from a computational point of view.

\begin{lemma} 
The mapping $\Phi \colon  G_{B_\alpha} \rightarrow H , \sigma \mapsto \ell_i^{-1}$ for $\sigma (\alpha) = \alpha_i = \ell_i (\alpha_1)$, is a surjective homomorphism with kernel $G_\alpha$.
\end{lemma}
\begin{proof}
Let $\sigma$, $\tau \in G_{B_\alpha}$ with $\Phi (\sigma) = \ell_i^{-1}$ and  $\Phi (\tau) = \ell_{j}^{-1}$. Then $\sigma \circ \tau (\alpha) = \sigma ( \ell_j (\alpha)) = \ell_j (\sigma \alpha) = (\ell_j\circ \ell_i) (\alpha)$. Thus $\Phi(\sigma \tau) = (\ell_j \circ \ell_i)^{-1 }= \ell_i^{-1} \circ \ell_j^{-1} = \Phi(\sigma) \Phi(\tau)$. Each $\ell_i$ defines an $\ff(t)$-automorphism on $\ff(\alpha)$, extends to $L$ and maps $B_\alpha$ to $B_\alpha$. Thus $\Phi$ is surjective. 
Finally, $\sigma$ is in the kernel if and only if $\sigma (\alpha) = \alpha$. This is if and only if $\sigma$ is in $G_\alpha$.
\end{proof}

Now let $s = n$. Then we have $B_\alpha = Z$ and $L = \ff(\alpha) = L^{G_\alpha}$. Thus $G_{B_\alpha} = G$ and $G_\alpha = 1$. In this case $\Phi$ is an isomorphism between $G$ and $H$ and we can compute all minimal blocks by the algorithm of \cite{atk75} in polynomial time. 
Before we consider the other cases we will see how to compute the appropriate $h$ from a block $B$. Instead of following \cite{zip91}, we use the following new result.

\begin{lemma}
\label{lem:comph}
Let $B$ be a block and $h$ be the right component of a decomposition of $f$ corresponding to $B$. Then $h(x) - h(\alpha) =  \prod_{\gamma \in B} (x - \gamma)$.
\end{lemma}
\begin{proof}
The block $B$ corresponds to the intermediate field $L^{G_B}$ and by \ref{thm:bij} there is a decomposition of $f$ with right component $h$ such that $L^{G_B} = \ff(h(\alpha))$. Set $\lambda = h(\alpha)$. Then the minimal polynomial of $\alpha$ over $\ff(\lambda)$ is $h - \lambda$.
Since $\alpha$ is in $B$ and both polynomials have the same degree, it is sufficient to show that $\prod_{\gamma \in B} (x - \gamma)$ is in $\ff(\lambda)[x]$. Let $\sigma$ be in $G_B$. Then $\sigma (B) = B$ and therefore $\sigma ( \prod_{\gamma \in B} (x - \gamma)) =  \prod_{\gamma \in B} (x - \sigma \gamma) =  \prod_{\gamma \in B} (x - \gamma)$. Since $\ff(\lambda) = L^{G_B}$ this proves that $\prod_{\gamma \in B} (x - \gamma)$ is in $\ff(\lambda)[x]$.
\end{proof}

Note that $\prod_{\gamma \in B} (-\gamma)$ is the constant term of $h(x) - h(\alpha)$. Since $h$ is normal, we get $h = \prod_{\gamma \in B} (x - \gamma)  - \prod_{\gamma \in B} (-\gamma)$, as explicit formula.

\begin{example}
Let $p$ be an odd prime and $\ff$ be a finite field of characteristic $p$. Let $f = x^2 \circ (x^p - x)$, $a$ be an element of the prime field $\mathbb{F}_p$ of $\ff$ and $\zeta$ be either $1$ or $-1$. Then $f(\zeta x + a) = (\zeta^p x^p + a^p - \zeta x - a)^2 = f(x)$. Thus, for a root $\alpha$ of $f - t$ also $\zeta \alpha + a$ is a root of $f$. Thus we have $2p$ roots of $f-t$ in $\ff(\alpha)$ and therefore $\ff(\alpha) \mid \ff(t)$ is Galois. Its Galois group is isomorphic to $\{ (\zeta x + a) \mid \zeta \in \{-1, 1\} , a \in \mathbb{F}_p\} \cong \mathbb{F}_p \rtimes \mathbb{Z}/2\mathbb{Z} \cong D_{2p}$. The dihedral group $D_{2p}$ has one subgroup of order $p$ and $p$ subgroups of order two. Hence $f$ has $p+1$ decompositions. A block with two elements is of the form $\{ \alpha, -\alpha + a\}$. Then $h(x) - h(\alpha) =(x - \alpha)( x- ( -\alpha + a)) = x^2 - ax - (\alpha^2 - a\alpha)$ and we have found the right component of a decomposition of $f$, namely $h=x^2 - ax$. 
\end{example}

Let $1 < s < n $.  Then the induced action of $G_{B_\alpha}$ on $B_\alpha$ is determined by the action of $H$ on $B_\alpha$, since $G_\alpha$ acts trivial on $B_\alpha$. 
If there are minimal blocks of $G_{B_\alpha}$ containing $\alpha$, one can find all of them in polynomial time (by the above mentioned algorithm of Atkinson).

\begin{lemma}
\label{lem:blocksinB}
If \ $\Lambda$ is a minimal block of $G_{B_\alpha}$, then $\Lambda$ is a minimal block of $G$.
\end{lemma}
\begin{proof}
Assume $\sigma (\Lambda)  \cap \Lambda \neq \emptyset $ for some $\sigma \in G$. Since $\Lambda \subseteq B_\alpha$ and $\sigma (\Lambda ) \subseteq \sigma (B_\alpha ) $, we get $\sigma (\Lambda)  \cap  \Lambda \subseteq  \sigma( B_\alpha) \cap B_\alpha \neq \emptyset$. Thus $\sigma (B_\alpha ) = B_\alpha$, which means that $\sigma$ is in $G_{B_\alpha}$. Since $\Lambda$ is a block of $G_{B_\alpha}$, we have $\sigma (\Lambda)  = \Lambda$. If $B \subseteq \Lambda$ is a nontrivial block of $G$, then $B$ is a block of $G_{B_\alpha}$ and thus $B = \Lambda$.
\end{proof}

Thus, we can easily compute all minimal blocks that are contained in $B_\alpha$. Note that if there is no nontrivial Block of $G_{B_\alpha}$, then $B_\alpha$ is a minimal block of $G$. 

\begin{example}
\label{ex:addpoly3}
Let $p=3$  and $f=x^9 -x$ over $\mathbb{F}_3$. Let $\alpha$ be a root of $f-t$. One checks that 
\begin{eqnarray*}
f(x)-f(\alpha) & = &  (x -\alpha)(x -\alpha + 1) (x -\alpha -1)  \\ 
& & (x^2 + \alpha x + \alpha^2 + 1) (x^2 + (\alpha + 1)x + \alpha^2 - \alpha -1)\\
& & (x^2 + (\alpha -1)x + \alpha^2 + \alpha -1)
\end{eqnarray*}
is the factorization of $f-t$ into irreducible polynomials over $\mathbb{F}_3(\alpha)$. As shown above $\{ \alpha, \alpha - 1, \alpha+1\}$ forms a block. Thus for $h(x) - h(\alpha) = (x -\alpha)(x -\alpha + 1) (x -\alpha -1) = x^3 -x -(\alpha^3 -\alpha)$ we have that $h = x^3 -x$ is the right component of a decomposition of $f$. One calculates that the corresponding left component is  $g=x^3 +x$.
\end{example}

If $s = 1$ then $B_\alpha = \{\alpha\}$ is trivial. Thus the method above does not apply. But also if $1 < s < n $, there could be minimal blocks $\Lambda$ with $\Lambda \cap B_\alpha = \{\alpha\}$. Thus, let $s < n$. Then $\ff (\alpha)$ is not Galois and we have $G_\alpha \neq 1$. Hence $G$ is not regular and the following theorem applies.
\begin{theorem}
Let $G$ be a finite permutation group on $Z$, which is transitive and not regular. Then $G$ is primitive if and only if 
for all distinct $\alpha$ and $\beta$ in $Z$ we have $ \langle G_\alpha, G_\beta \rangle  = G$.
\end{theorem}
\begin{proof}
Let $G$ be imprimitive and $\Lambda$ be a nontrivial block with $\alpha$, $\beta \in \Lambda$ and $\alpha \neq \beta$. Then $G_\alpha$, $G_\beta \subseteq G_\Lambda$. Thus by \ref{thm:wiel}, we get $\langle G_\alpha, G_\beta\rangle  \subseteq G_\Lambda \neq G$, since $\Lambda \neq Z$.

Now let $\langle G_\alpha, G_\beta\rangle \neq G$ for some $\alpha \neq \beta$. Then $\Lambda = \langle G_\alpha, G_\beta \rangle (\alpha)$ is a block $\neq Z$ as shown in the proof of \ref{thm:wiel}. If $\Lambda$ is nontrivial, then we are done. Thus assume $\Lambda = \{ \alpha \}$.
Then we have $\sigma (\alpha) = \alpha$ for all $\sigma \in G_\beta$. Thus $ G_\beta \subseteq G_\alpha$. 
Since $|G_\alpha| = |G_\beta|$ we have $G_\alpha = G_\beta$. But then $\alpha$, $\beta \in B_\alpha$ and $B_\alpha$ is trivial only if $B_\alpha = Z$. Since $G$ is not regular there is $1\neq \sigma \in G_\alpha$ and $\gamma$ such that $\sigma(\gamma) \neq \gamma$. Thus $\gamma \notin B_\alpha$. Hence $B_\alpha \neq Z$, which is thus a nontrivial block.
\end{proof}

\begin{proposition} 
\label{satz:lam} 
Let $\Lambda$ be a minimal block of $G$ with $\alpha \in \Lambda$ and $\Lambda \cap B_\alpha = \{\alpha\}$. Then for all $\beta \in \Lambda$ distinct from $\alpha$ the orbit $\langle G_\alpha , G_\beta \rangle  (\alpha)$ equals $\Lambda$. 
\end{proposition}
\begin{proof}
Let $\beta \neq \alpha$ be in $\Lambda$. Clearly $\langle G_\alpha , G_\beta \rangle  (\alpha) = \{ \sigma (\alpha ) \mid \sigma \in \langle G_\alpha , G_\beta \rangle  \} \subseteq \Lambda$. Now if $G_\alpha = G_\beta$ then $\beta$ would be fixed by $G_\alpha$ and thus $\beta \in B_\alpha$, which is a contradiction to the assumption. Thus we have $G_\alpha \neq G_\beta$ and therefore $|\langle G_\alpha , G_\beta \rangle  (\alpha)| > 1$. Thus $\langle G_\alpha , G_\beta \rangle (\alpha)$ is a nontrivial block contained in $\Lambda$, which implies equality by the minimality of $\Lambda$.
\end{proof}

\begin{lemma}
\label{lem:gab}
Let $\fmt = \prod_{i=1}^r \psi_i$ be a factorization of $\fmt$ into irreducible factors $\psi_i$ over $\ff(\alpha)$. Let $\beta \in Z$ and $j$ be such that $\psi_j (\beta) = 0$. Then $G_\alpha (\beta)  = \{ \gamma \in Z \colon \psi_j (\gamma ) = 0\}$.
\end{lemma}
\begin{proof}
If $\psi_j = (x - \beta)$ then $\beta$ is in $\ff(\alpha)$ and is thus fixed by $G_\alpha$. Hence the claim holds. Thus let $\deg(\psi_j) \geq 2$. Each $\sigma \in G_\alpha$ acts trivial on $\ff(\alpha )$. Thus $\sigma ( \psi_j) = \psi_j$ and $\psi_j( \sigma (\beta ) ) = \sigma (\psi_j (\beta )) = 0$. Hence $G_\alpha (\beta)  \subseteq \{ \gamma \in Z \colon \psi_j (\gamma ) = 0\}$. 

For the other direction let $\gamma$ be such that $\psi_j (\gamma) =0$. Let $M$ be the splitting field of $\psi_j$ over $\ff(\alpha)$. Then there is $\sigma$ in $\mbox{Gal} (M \mid \ff(\alpha))$ such that $\sigma (\beta) = \gamma$. Since $\ff(\alpha) \subseteq M \subseteq L$ we have that $\sigma$ extends to an automorphism in $\mbox{Gal}(L \mid \ff(\alpha)) = G_\alpha$. Thus $\gamma$ is in $G_\alpha (\beta)$.
\end{proof}

Fix $\nu>s$ and $\beta$ such that $\beta$ is a root of $\psi_\nu$. Note that $\beta \notin B_\alpha$ and thus $\langle G_\alpha , G_\beta \rangle  (\alpha)$ is a block, which is minimal if there is a minimal block containing $\alpha$ and $\beta$. Let $\sigma$ be in $G$ such that $\sigma(\alpha) = \beta$ and set $\psi^*_i = \sigma(\psi_i)$ for all $ 1\leq i \leq r$. Then the polynomials $\psi^*_i \in \ff(\beta) [x]$ are the polynomials $\psi_i$ with $\beta$ substituted for $\alpha$ and the irreducible factors of $\fmt$ over $\ff(\beta)$ are precisely the polynomials $\psi^*_i$. Note that if $\gamma$ is a root of $\psi^*_j$ we have $G_\beta (\gamma) =  \{ \gamma' \in Z \colon \psi^*_j (\gamma' ) = 0\}$ by the previous lemma. 

\begin{proposition}
\label{lem:lambdagcd}
Consider the bipartite graph $\Gamma_\beta$ with the set of vertices consisting of $\psi_i$ and $\psi^*_i$ for $ 1\leq i \leq r$ and with an undirected edge between $\psi_i$ and $\psi^*_j$ if $\gcd (\psi_i , \psi^*_j ) \neq 1$. Let $C_\beta$ be the the set of roots of those $\psi_i$ that are connected to $\psi_1$. Then $\langle G_\alpha , G_\beta \rangle  (\alpha) = C_\beta$.
\end{proposition}
\begin{proof}
Each element $\gamma$ of $\langle G_\alpha , G_\gamma \rangle  (\alpha)$ is of the form $\sigma_u \ldots \sigma_2 \sigma_1(\alpha)$ with $\sigma_i$ in $G_\alpha$ or in $G_\beta$. We prove by induction on $u$ that $\gamma$ is in $C_\beta$. The induction basis is the fact that $\alpha \in C_\beta$. 
For the induction step let $\gamma = \sigma_{u-1} \ldots \sigma_2 \sigma_1 (\alpha)$ be in $C_\beta$. Then there is some $i$ such that $\psi_i(\gamma) = 0$ and $\psi_i$ is connected to $\psi_1$. We distinguish two cases: First we have $\sigma_u \in G_\alpha$. Since $\sigma_u (\psi_i) = \psi_i$, we have $0=\sigma_u (\psi_i (\gamma)) = \psi_i (\sigma_u\gamma)$. Thus also $\sigma_u(\gamma) \in C_\beta$. In the second case we have $\sigma_u \in G_\beta$. Let $j$ such that $\psi^*_j (\gamma) = 0$. Then there is an edge between $\psi^*_j$ and $\psi_i$ and thus $\psi^*_j$ is connected to $\psi_1$. Since $\sigma_u (\gamma) \in G_\beta(\gamma)$ we have that also $\sigma_u (\gamma)$ is a roots of $\psi^*_j$. Hence if $\sigma_u (\gamma)$ is a root of $\psi_k$ we have $\gcd(\psi_k , \psi^*_j ) \neq 1$ and thus $\sigma_u (\gamma) \in C_\beta$. 

For the other direction, let $\gamma$ be in $C_\beta$ and $i$ such that $\psi_i(\gamma ) = 0$.  Then there is a path $P$ form $\psi_1$ to $\psi_i$, say $P = (\psi_1, \psi^*_\ell), \ldots, (\psi_k, \psi^*_j) (\psi^*_j, \psi_i)$. By an induction argument one can assume that the roots of $\psi_k$ are already in $\langle G_\alpha , G_\gamma \rangle  (\alpha)$. Since $\gcd(\psi^*_j, \psi_i) \neq 1$ there is $\beta'$ such that $\psi^*_j(\beta') = 0$ and $\psi_i(\beta')= 0$. Thus, by \ref{lem:gab} we get $\gamma \in G_\alpha (\beta')$ and $\beta ' \in G_\beta (\alpha')$ where $\alpha'$ is a common root of $\psi_k$ and $\psi^*_j$. Then $\alpha' \in \langle G_\alpha , G_\beta \rangle (\alpha)$ and $\gamma = \sigma_1 \sigma_2 (\alpha')$ for $\sigma_1 \in G_\beta$ and $\sigma_2 \in G_\alpha$. Hence $\gamma$ is in $\langle G_\alpha , G_\beta \rangle  (\alpha )$.
\end{proof}

Now let $\Lambda$ be a minimal block containing $\alpha$ and $\beta$. In case $\Lambda \subseteq B_\alpha$ we saw that one can calculate $\Lambda$ directly and therefore calculate $h$ by \ref{lem:comph}. Otherwise by \ref{satz:lam} we have $\Lambda = \langle G_\alpha , G_\beta \rangle  (\alpha)$. Then as seen in \ref{lem:lambdagcd} one can calculate $h$ by $ h(x) - h(\alpha) = \Pi_{\gamma \in \Lambda} (x - \gamma) = \Pi \psi_i$, where the last product is taken over all $\psi_i$ that are connected to $\psi_1$ in $\Gamma_\beta$.

\begin{contexample}{ex:addpoly3}
Let us continue to find decompositions of $f = x^9 - x$. Let 
\begin{align*}
\psi_1 &= x^2 + \alpha x + \alpha^2 + 1, \\
\psi_2 &= x^2 + \alpha x + x + \alpha^2  - \alpha - 1 \mbox{ and} \\ 
\psi_3 &= x^2 + \alpha x  - x + \alpha^2 + \alpha - 1 .
\end{align*} 
Then as before we have $f - t = (x - \alpha)(x - \alpha + 1)( x - \alpha -1)\psi_1\psi_2\psi_3$. Now let $\zeta$ be a root of $x^2 +x -1$ in $\mathbb{F}_9$ and note that we have then 
\begin{align*}
\psi_1 &= (x - (\alpha + \zeta + 1))(x - (\alpha - \zeta - 1)),\\
\psi_2 &= (x - (\alpha + \zeta - 1))(x - (\alpha - \zeta )) \mbox{ and}\\
\psi_3 &= (x - (\alpha + \zeta))(x - (\alpha - \zeta + 1)) .
\end{align*}
Let $\beta_1 = \alpha + \zeta + 1$. Then we have that $\psi_1$ with $\beta_1$ substituted for $\alpha$ is $\psi^*_1 = (x - \alpha)(x - (\alpha -\zeta -1))$ and thus $C_{\beta_1} = \{ \alpha , \alpha + \zeta + 1 , \alpha - \zeta - 1\}$ is a minimal block. 
We get $h(x) - h(\alpha) = (x- \alpha)\psi_1 = x^3 +x - (\alpha^3 + \alpha)$ and thus the corresponding decomposition has right component $x^3 + x$. Then the corresponding left component is $x^3 - x$.

Now for $\beta_2 = \alpha + \zeta -1$ we get $\psi^*_3 = (x- \alpha)(x  -(\alpha- \zeta -1))$. Thus $C_{\beta_2} = C_{\beta_1}$. In the same way $\alpha - \zeta +1$ does not yield any further block. Therefore all in all $f$ has exactly two decompositions over $\mathbb{F}_3$.

Note that going to the extension $\mathbb{F}_9$ of $\mathbb{F}_3$ unveils more structure. Actually, $f$ has four decompositions over $\mathbb{F}_9$ as we will see in \ref{ex:add}.
\end{contexample}

\subsection{The algorithm}
\cite{zip91} describes loosely an algorithm that computes decompositions of rational functions. The following is a concrete description of an algorithm that computes all minimal decompositions of a polynomial, whose derivative does not vanish. It is mainly based on \cite{zip91} and on \cite{lanmil85}. The runtime estimation in \ref{thm:impl} is new.

\begin{figure}[h!]
\begin{algorithm}{minDec}[Computing minimal decompositions]
\item A monic polynomial $f \in \ff [x]$ of degree $n$ with $f' \neq 0$.
\item A list of decompositions $(g, h)$ of $f$. This list is empty if $f$ is indecomposable.
\begin{block}
	\item Set $List = \{\}$ and let $\ff(\alpha)$ be the rational function field in $\alpha$. 
	\item \algolabel{factor} Factor $f(x) - f(\alpha)$ in $\ff (\alpha) [x]$ into $\prod_ {i = 1 } ^s ( x - \alpha_i) \cdot \psi_{s+1} \cdot \ldots \cdot \psi_r$ as in \ref{equ:factor}. 
	\begin{ifblock}
		{$s > 1$}
		\item Set $B_\alpha = \{ \alpha_i \mid 1 \leq i \leq s\}$ and $H = \{ \ell_i \mid 1 \leq i \leq s\}$ where $\alpha_i = \ell_i(\alpha)$. 
		\item \algolabel{atkinson} Set $AtkinsonBlocks = $ \texttt{Atkinson}($H$, $B_\alpha$, $\alpha$).
		\begin{forblock}
			{$\Lambda \in  AtkinsonBlocks$ with $ |\Lambda | < n$}
			\item \algolabel{h1} Compute $h(x) = \prod_{\gamma \in \Lambda} (x - \gamma) - \prod_{\gamma \in \Lambda}  (-\gamma)$.
			\item \algolabel{g1} Compute $g$ such that $f = g \circ h$.
			\item Attach $(g, h)$ to $List$.
		\end{forblock}
	\end{ifblock}
	\begin{forblock}
		{$ \nu \in \{s+1, \ldots , r \}$}
		\item Let $\beta$ be a root of $\psi_\nu$ and let $\psi^*_i$ be $\psi_i$ with $\beta$ substituted for $\alpha$, for all $1 \leq i \leq r$.
		\item \algolabel{gcd} Compute the graph $\Gamma_\beta$ as in \ref{lem:lambdagcd}.
		\item \algolabel{inu} Compute $I_\nu = \{i  \colon \psi_i \mbox{ is connected to } \psi_1 \mbox{ in } \Gamma_\beta \}$.
		\begin{ifblock}
			{$I_\nu \neq \{1, \cdots, r \}$}
			\item \algolabel{h2} Compute $h(x) - h(\alpha) = \prod_{i \in I_\nu} \psi_i$, where $h(x)$ is in $\ff [x]$ normal.
			\item \algolabel{g2} Compute $g$ such that $f = g \circ h$.
			\item Attach $(g, h)$ to $List$.
		\end{ifblock}
	\end{forblock}
	\item \RETURN $List$.
\end{block}
\end{algorithm}
\end{figure}

\ref{minDec} calls a subroutine \texttt{Atkinson}($G$, $Z$, $\alpha$) which returns a list of all minimal blocks of the permutation group $G$ on $Z$ that are containing $\alpha$. If $G$ is primitive this list consists of $Z$ only. 

\begin{theorem}
\ref{minDec} correctly computes all minimal decompositions of $f$.
\end{theorem}
\begin{proof}
Let $(g, h)$ be a minimal decomposition and let $\Lambda$ be the corresponding block. Then either $\Lambda \subseteq B_\alpha$ or $\Lambda \cap B_\alpha = \{\alpha\}$. In the first case $\Lambda$ is computed in Step \short\ref{minDec-atkinson}, by \ref{lem:blocksinB}. Then $h$ is recovered from $\Lambda$ in Step \short\ref{minDec-h1}, by \ref{lem:comph}.
In the second case let $\beta \in \Lambda \setminus \{\alpha\}$ and $\nu$ such that $\psi_\nu (\beta) = 0$. By \ref{lem:lambdagcd} we have $\Lambda = C_\beta = \{\gamma \colon \exists i \in I_\nu \colon \psi_i (\gamma) = 0 \}$, where $I_\nu$ is computed in Step \short\ref{minDec-inu}. Then in Step \short\ref{minDec-h2} we have $\prod_{i \in I_\nu} \psi_i=\prod_{\gamma \in \Lambda} (x - \gamma) = h(x) - h(\alpha)$, from which we can recover $h$. 
\end{proof}

Note that $C_\beta$ is a block even if there is no minimal block containing $\alpha$ and $\beta$. Then either $C_\beta$ is minimal and $\beta \notin C_\beta$ or $C_\beta$ is not minimal and contains minimal blocks. An algorithm that computes \emph{only} minimal decompositions (and outputs each decomposition only once) should keep track of this.

For the runtime consideration let $\ff$ be a field over which one can factor bivariate polynomials in polynomial time (these include, for example, finite fields). The gcd computation in Step \short\ref{minDec-gcd} can be done by computing the resultant, which can be done in polynomial time. Thus the algorithm can be implemented with a polynomial runtime. 
This was already remarked by \cite{zip91}. The following runtime estimation for finite fields is new.

\begin{theorem}
\label{thm:impl}
Let $\ff$ be a finite field, $c \geq 3$ be a natural number and $n$ be the degree of the input polynomial $f$.
Denote the complexity of multiplying two polynomials over $\ff$ of degree at most $n$ by $M(n)$ and let $q$ be the size of $\ff$.
Then there is an implementation of \ref{minDec} that takes an expected number of  $\mathcal{O}^\sim (c n^4M(n)^2 \log(q))$ 
operations in $\ff$ with an error probability of at most $n^3(4n)^{-c}$.
\end{theorem}
\begin{proof}
The factorization in Step \short\ref{minDec-factor} can be done in $\mathcal{O}^\sim (n^{\omega +1})$, where $2 \leq \omega \leq 3$ is the matrix multiplication exponent, see \cite{boslec04} and \cite{lec08}. 
Atkinson's algorithm runs in $\mathcal{O} (n^3)$, see \cite{atk75}. \cite{but92} improved the runtime of Atkinson's algorithm to $\mathcal{O} (n^2 \log n)$.
In Step \short\ref{minDec-g1} and \short\ref{minDec-g2} for each right component $h$ the appropriate left component $g$ can be computed in $\mathcal{O} (M(n) \log{n})$ by the generalized Taylor expansion, see \cite{gatger99}. Since there are at most $s$ minimal blocks computed by the algorithm of Atkinson, Step \short\ref{minDec-g1} is called at most $s$ times. Step \short\ref{minDec-g2} is called at most $r -s$ times. Thus we get $\mathcal{O} (r M(n)\log n)$ for this part.

To compute the graph in Step \short\ref{minDec-gcd} we need at most $r^2$ gcd computations. We have to compute $r-s$ such graphs. Thus in total we have at most $(r-s)r^2 \leq n^3$ such gcd computations. Since the field arithmetic of $\ff (\alpha, \beta)$ is quite costly, one should use a modular algorithm that checks if two polynomials in $\ff(\alpha, \beta) [x]$ are coprime. For example one could use \ref{algo:gcd} below. We will prove that it has expected runtime $\mathcal{O}^\sim(n M(n)^2 \log(q))$ and an error probability of at most $(4n)^{-1}$. 
If we repeat this coprimality check $c$ times we get for all $n^3$ computations an expected runtime of $\mathcal{O}^\sim (c n^4M(n)^2 \log(q))$, which dominates the runtime of the other computations, and 
an error probability of at most $1 - (1- (4n)^{-c})^{n^3} \leq 1 - (1 -n^3(4n)^{-c}) = n^3(4n)^{-c}$, by the Bernoulli inequality.
This finishes the proof of \ref{thm:impl}.
\end{proof}

\begin{figure}[h!]
\begin{algorithm}{algo:gcd}[Coprimality in $\ff(\alpha, \beta){[x]}$ ]
\item An irreducible polynomial $G \in \ff[x,y]$ of total degree at most $n$, that defines $\ff(\alpha, \beta)$ by $G(\alpha, \beta) = 0$, and two polynomials $g$, $h \in \ff[\alpha, \beta, x]$ that are monic in $x$. 
\item $True$ / $False$.
\begin{block}
	\item Let $K' \mid \ff$ be a field extension of $\ff$ with $[K' \colon \mathbb{F}_p] \geq 4\log(16n)$.
	\item Randomly choose $a$ in $K'$.
	\item Compute a root $b$ of $G(a, y)$ in an extension $K$ of $K'$.
	\item Compute $r = \mbox{res}(g(a,b,x), h(a,b,x))$.
	\item \RETURN $True$ if $r \neq 0$ and $False$ else.
\end{block}
\end{algorithm}
\end{figure}

\begin{lemma}
Let $F$ be a finite field of size $q$. If the total degree of all input polynomials is bounded by $n$, then \ref{algo:gcd} takes an expected number of $\mathcal{O}^\sim(M(n)^2n\log(q))$ operations in $\ff$. It returns $True$ only if $g$ and $h$ are coprime. If $g$ and $h$ are coprime the algorithm returns $False$ with probability at most $(4n)^{-1}$.
\end{lemma}
\begin{proof}
First note that for $(a,b) \in K^2$ with $G(a,b) = 0$ the maximal ideal $(\alpha - a , \beta - b)$ in $K[\alpha, \beta]$ gives rise to a place $P$ in $K(\alpha, \beta)$. Substituting $a$ and $b$ for $\alpha$ and $\beta$, respectively, is the same as reducing modulo $P$. Thus the degree of $P$ is one. If on the other hand $P$ is a place in $K(\alpha, \beta)$ of degree one such that it is neither a pole of $\alpha$ nor of $\beta$, then there exists $a$, $b \in K$ such that $\alpha \equiv a \mod P$ and $\beta \equiv b \mod P$. Then we have $P \cap K[\alpha, \beta] = (\alpha - a, \beta - b)$. Denote such a places by $P_{a,b}$. 

Since the leading coefficient of $g$ in $x$ is one, and therefore not zero modulo $P$, we have $\mbox{res}(g(a,b,x), h(a,b,x)) = 0$ if and only if $\mbox{res}(g, h)\equiv 0 \mod P$ (see Lemma 6.25 in \cite{gatger99}). Assume $\rho = \mbox{res}(g,h) \in K[\alpha, \beta]$ is not zero. Let $A = | \{ (a,b) \in K' \times K \colon G(a,b) =0\} |$ and $ B = | \{ (a,b) \in A \colon \rho \equiv 0 \mod P_{a,b} \} |$. Then the probability of $\rho$ being zero modulo $P$ is $B/A$.

We have $B \leq |\{ (a,b) \in K^2 \mid G(a,b) =0 =R(a,b)\}|$, where $R$ is a representative of $\rho$ in $\ff [x,y]$ of degree less then $2n^2$. 
If $G$ would divide $R$, then $\rho$ would be zero. Since $G$ is irreducible we get $\gcd(G, R) = 1$ and thus by B\'ezout's Theorem we have $B \leq \deg( G ) \deg (R) \leq 2n^3$. 

Since $\alpha$ and $\beta$ can have at most $n$ poles we have $A \geq N - 2n$, where $N$ is the number of places in $K'(\alpha, \beta)$ of degree one. 
By the Hasse--Weil Bound (see Theorem V.2.3 in \cite{sti93}) we get $N \geq q' + 1 - 2g q'^{1/2}$, where $q'$ is the size of $K'$ and $g$ is the genus of $K'(\alpha, \beta)$. By the degree assumption on $K'$ we have that $q' \geq p^{4\log(16n)} \geq 16n^4$. By the Riemann Inequality (see Corollary III.10.4 in \cite{sti93}) the genus is bounded by $([K'(\alpha, \beta) \colon K'(\alpha)] - 1)([K'(\alpha, \beta)\colon K'(\beta)] - 1) \leq (n-1)^2$. Thus we have $A \geq N - 2n \geq q' + 1 -2g q'^{1/2} - 2n\geq q'^{1/2} ( q'^{1/2} - 2g) - 2n \geq 4n^2 ( 4n^2 - 2(n-1)^2) - 2n \geq 8n^4$. Hence $B/A \leq (2n^3) /(8n^4) = 1/(4n)$, which gives us the claimed error bound. 

The resultant computation takes $\mathcal{O}((M(n)+n) \log n)$ operations in $K$ (see Corollary 11.16 in \cite{gatger99}).
Finding a root of $G(a,y)$ in $K$ takes an expected number of $\mathcal{O}(M(n) \log n  \log (n \tilde{q}))$ operations in $K$, where $\tilde{q}$ is the size of $K$ (see Corollary 14.16 in \cite{gatger99}). 
We have $\log(\tilde{q}) \leq [K \colon \mathbb{F}_p] = [K \colon \ff] [\ff \colon \mathbb{F}_p]$. Unfortunately the degree of $K$ over $K'$ can only be bounded by $\deg_y(G(a,y)) \leq n$. Thus $[K\colon \ff]$ is in $\mathcal{O}( n\log n)$ and finding a root takes $\mathcal{O}^\sim (M(n)n \log(q))$ operations in $K$. Arithmetic in $K$ costs us $\mathcal{O}(M(n\log n)\log (n\log n))$ operations in $\ff$. Hence with omitting the log factor the expected runtime is in $\mathcal{O}^\sim(M(n)^2n \log(q))$. 
\end{proof}

\subsection{An upper bound}
\label{subsec:ub}
Now we will deduce two sharp upper bounds on the number of minimal decompositions of a polynomial. These bounds coincide partly with results in \cite{gatgie10}.

Let $B$ and $B'$ be minimal blocks. By \ref{lem:inters} their intersection is a block and therefore trivial. Hence, the minimal blocks minus $\{ \alpha \}$ are distinct sets in $Z \setminus \{ \alpha \}$. Therefore the sum of the cardinality of all minimal blocks minus $\{\alpha\}$ is less than $n-1$. Since the cardinality of a block equals the degree of the right component of the corresponding normal decomposition, we get the following results.

\begin{corollary}
Let $f$ be a decomposable polynomial of degree $n$ with $f' \neq 0$. 
\begin{enumerate}
\item Let $d$ divide $n$. Then there are at most $(n - 1)/(d - 1)$ minimal decompositions $(g,h)$ of $f$ with $\deg(h) = d$.
\item Let $q$ be the smallest prime divisor of $n$. Then there are at most $(n-1)/(q-1)$ minimal decompositions of $f$.
\end{enumerate}
\end{corollary}

\begin{example}
\label{ex:add}
Let $p$ be the characteristic of $\ff$ and let $f$ be a separable additive polynomial of degree $p^r$ with $r \geq 2$, that is $f$ is of the form $\sum_{i=0}^r a_i x^{p^i}$ with $a_0 \neq 0$. Furthermore assume that $f$ splits completely over $\ff$. Then the roots of $f$ form a group $G \subseteq \ff$ which is isomorphic to $(\mathbb{Z}/p\mathbb{Z})^r$. If $\alpha$ is a root of $\fmt=f - t$ then so is $\alpha + a$ for all roots $a$ of $f$. Thus $\fmt$ is Galois and its Galois group is isomorphic to $G$. But $G$ has exactly $(p^r - 1)/(p-1)$ subgroups of order $p$. Thus $f$ has exactly $(p^r-1)/(p-1)$ minimal decompositions. This shows that both bounds are sharp.
\end{example}

Another class of examples, which also shows that these bounds are sharp, is discussed in the next section.




\section{A taxonomy}
\label{sec:conj}

\cite{zan93} proved Ritt's second Theorem over arbitrary fields with the assumption that the derivatives of the right components are nonzero and that their degrees are relatively prime. Unfortunately this result is not applicable when the degrees are equal -- these are so called equal-degree collisions. A classification for a special case of equal-degree collision was proposed in \cite{gatgie10}, Conjecture 6.7. 
The main result of this section (\ref{thm:gen}) is a proof of this conjecture. 
The idea is to adapt the proofs of Ritt's second Theorem from \cite{dorwha74} and \cite{zan93}. In both papers the classification is obtained from studying ramification in a rational function field.

\subsection{Preliminaries}
\label{sec_prel}

First we will state some facts from the ramification theory of function fields. For more details see, for example, \cite{sti93}. 
Let $\ff$ be a field of characteristic $p$ and $K$ be its algebraic closure. Let $f$ be a normal polynomial over $\ff$ of degree $n$ such that $f' \neq 0$ and $p \mid n$. As in \ref{sec:dec} let $t$ be transcendental over $K$. Then $\fmt =f-t$ is irreducible and separable over $K(t)$. Let $\alpha$ be a root of $\fmt$. Then $K(\alpha)$ is an extension of $K(t)$ of degree $n$. Both function fields are rational (that is, of genus 0). Thus, each finite place $P$ in $K(t)$ corresponds to a monic and irreducible polynomial in $K[t]$ (see Section I.2 in \cite{sti93}). This polynomial is linear, since $K$ is algebraically closed, say of the form $t - c$ with $c \in K$. In $K(\alpha)$ we have $t - c = f(\alpha) - c = \prod g_i^{e_i}(\alpha)$, where $\prod g_i^{e_i}$ is a factorization of $f - c$ into irreducible factors $K[x]$. The $g_i$ are linear and correspond to places $\priP_i$ in $K(\alpha)$. Then $\priP_i^{e_i}$ divides $P$. Since $\sum e_i = \deg f = [K(\alpha) \colon K(t)]$ we obtain a  decomposition $P = \prod \priP_i^{e_i}$. Thus the multiplicities of $f - c$ correspond to the ramification indices of $P$, that is $e_i = e(\priP_i \mid P)$. 

Later in this section we wish to have certain multiple roots at the right "place". One can achieve this by conjugation of $f$ (see \ref{def:lcomp}).
If $w$ has multiplicity $m$ in $f - f(w)$ then the conjugate $(x - f(w) ) \circ f \circ (x + w)$ has a root at $0$ with multiplicity $m$.

We will make use of the notion of the different exponent $d(\priP \mid P)$ of a place $\priP \mid P$. 
Mainly we need the following facts about the different exponent (for a definition and further facts see Section III.4 in \cite{sti93}): A place $\priP$ is unramified over $P$ if and only if $d(\priP \mid P) = 0$. If $\priP$ is tamely ramified over $P$, then $d(\priP \mid P) = e(\priP \mid P) - 1$. Since $K$ is algebraically closed, the relative degree of $\priP \mid P$ equals one. Thus, if $P$ is tamely ramified, we get $\sum_{\priP \mid P} d(\priP \mid P) = n - \rho$, where $\rho$ is the number of places in $K(\alpha)$ lying over $P$.

The following results tell us more about the ramification in rational function fields over $K$. (These results are true for arbitrary fields, but they are needed here only for the algebraic closure of $\ff$.)
\begin{proposition}
The place at infinity in $K(t)$ is totally ramified in $K(\alpha)$.
\end{proposition}
\begin{proof}
See Proposition 3.2. in \cite{frimac69}.
\end{proof}

\begin{proposition}\label{diffexp}
Let $E \mid K(t)$ be a finite separable extension. Let $P$ be a place in $K(t)$ and $\priP$ be a place in $E$ which is totally ramified over $P$. Let $\pi$ be a prime element of $\priP$ and $\psi$ its minimal polynomial over $K(t)$. Then $d(\priP \mid P) = v_{\priP} (\psi ' (\pi))$, where $v_{\priP}$ is the valuation at $\priP$.
\end{proposition}
\begin{proof}
See Proposition III.5.12 in \cite{sti93}.
\end{proof}

\begin{lemma}
\label{lem:finiteRam}
Let $\Pinf$ be the infinite place of $K(t)$ and $\priPinf$ be the place in $K(\alpha)$ over $\Pinf$. Then $d(\priPinf \mid \Pinf) = 2n - 2 - \deg(f')$ and 
$$\sum_{\priP \textrm{ finite}} d(\priP \mid P) = \deg(f').$$
\end{lemma}
\begin{proof}
Since $\priPinf$ is totally ramified we can apply \ref{diffexp}. We have that $\alpha^{-1}$ is a primitive element of $\priPinf$. Let $\psi$ be the minimal polynomial of $\alpha^{-1}$. We have $0 = \alpha^{-n} (f(\alpha) - t) = \hat{f} (\alpha^{-1}) - t \alpha^{-n}$, with $\hat{f}$ being the reversal of $f$. Since $f$ is original we have $\deg (\hat{f}) < n$. Then $x^n - t^{-1}\hat{f} (x)$ is a monic polynomial, and since $[K(\alpha^{-1}) \colon K(t)] = n$, we get $\psi = x^n - t^{-1}\hat{f} (x)$. Thus we have $\psi ' = -t^{-1}\hat{f} ' (x)$ and \ref{diffexp} yields $d(\priPinf \mid \Pinf) = v_\infty (\psi ' (\alpha^{-1}))= v_\infty (-t^{-1} \hat{f} ' (\alpha^{-1})) = v_\infty (-t^{-1}) + v_\infty( \hat{f} ' (\alpha^{-1}))$. Since $t^{-1}$ is a primitive element of $\Pinf$ we have $v_\infty (-t^{-1}) = n$. Let  $\hat{a}_j$ be the coefficients of $\hat{f}$. Then by the strong triangle inequality we get $v_\infty( \hat{f} ' (\alpha^{-1})) \geq \min\{ v_\infty (j\hat{a}_j \alpha^{-(j-1)} ) \mid j\hat{a}_j \neq 0\} $ and equality since we have $v_\infty (j\hat{a}_j \alpha^{-(j-1)} )= j + 1 \neq i+1 = v_\infty (i\hat{a}_i \alpha^{-(i-1)} )$ for all $i \neq j$. The $(j-1)$-th coefficient of $\hat{f}'$ is nonzero if $p \nmid j$ and $ \hat{a}_j \neq 0$. But since $p\mid n$ this is the case if and only if $ p \nmid (n - j)$ and the $(n - j)$-th coefficient of $f$ is nonzero. 
Thus, the last nonzero coefficient in $\hat{f}'$ is the first nonzero coefficient in $f'$.
Hence $v_\infty( \hat{f} ' (\alpha^{-1})) = n - (\deg (f') + 1) - 1$ and therefore $d(\priPinf \mid \Pinf) = 2n - 2 - \deg(f')$.

By the Hurwitz Genus Formula we have $2g' - 2 = [K(\alpha) \colon K(t)] (2g -2) + \sum_\priP d(\priP \mid P)$, where $g$ and $g'$ is the genus of $K(t)$ and $K(\alpha)$, respectively (see Theorem III.4.12 in \cite{sti93}). In our case we have $g, g' =0$ and thus obtain $\sum_\priP d(\priP \mid P) = 2[K(\alpha) \colon K(t)] - 2 = 2n - 2$. By subtracting $d(\priPinf \mid \Pinf)$ we get  $\sum_{\priP \textrm{ finite}} d(\priP \mid P) = 2n - 2 - (2n - 2 - \deg(f') ) = \deg(f')$. 
\end{proof}

There is also an elementary proof for the last equation of \ref{lem:finiteRam} if there is no finite wildly ramified place; see the second proof of Lemma 2 in \cite{dorwha74}. 

\begin{lemma}
\label{lem:Dorey}
Let $M$ and $N$ be two intermediate fields of $K(\alpha) \mid K(t)$ such that $M N = K(\alpha)$ and let $\prip$ and $\priq$ be finite places in $M$ and $N$, respectively, over a place $P$ in $K(t)$. Let the ramification indices $e = e(\prip \mid P)$  and $\et = e(\priq \mid P)$ be not divisible by the characteristic of $K$. Then there are $\gcd(e, \et)$ places $\priP$ in $K(\alpha)$ which lie over $\prip$ and over $\priq$. Moreover for such a place we have  $e(\priP \mid P) = \emph{lcm}(e, \et)$.
\end{lemma}

This lemma is proven in \cite{dorwha74} with the assumption that the characteristic of $K$ is zero. The following proof is only slightly different, but does not use this assumption. 
\begin{proof}
At first we apply Abhyankar's Lemma (see Proposition III.8.9 in \cite{sti93}) and find that for a place $\priP$ in $K(\alpha)$, which lies over $\prip$ and over $\priq$, the ramification index $e(\priP \mid P)$ equals $\lcm(e, \et)$.

Then we proceed as in \cite{dorwha74}. Let $\widehat{K(t)}$, $\widehat{M}^\prip$, and $\widehat{N}^\priq$ be the completions of $K(t)$, $M$, and $N$ with respect to $P$, $\prip$, and $\priq$, respectively. For readability set $E = K(t)$ and $E^* = \widehat{K(t)}$. Note that $K(\alpha) \otimes_E \widehat{K(t)}$ is the direct product of the completions of $K(\alpha)$ with respect to the places over $P$ in $K(\alpha)$ (see Proposition II.8.3 in \cite{neu99}). Since $N \otimes_E M \cong N M = K(\alpha)$ we get $K(\alpha) \otimes_{M} \widehat{M}^\prip \cong N \otimes_E M  \otimes_{M} \widehat{M}^\prip = N \otimes_E \widehat{M}^\prip = N  \otimes_E (\widehat{K(t)} \otimes_{E^*} \widehat{M}^\prip) \cong (N \otimes_E \widehat{K(t)} )\otimes_{E^*} \widehat{M}^\prip \cong \bigoplus_\priq \widehat{N}^\priq \otimes_{E^*} \widehat{M}^\prip$. Thus $\widehat{N}^\priq \otimes_{E^*} \widehat{M}^\prip$ is the direct product of the completions of $K(\alpha)$ with respect to the places that lie over $\prip$ and $\priq$. These fields are of degree $\lcm(e, \et)$ and the $E^*$-dimension of $\widehat{N}^\priq \otimes_{E^*} \widehat{M}^\prip$ is $e \cdot \et$. Thus there are $e \cdot \et / \lcm(e, \et) = \gcd(e, \et)$ places over $\prip$ and $\priq$.
\end{proof}

\subsection{Decompositions of polynomials of degree $p^2$}
\label{sec_p2}
We have already seen an example of decomposable polynomials of degree $p^2$: additive polynomials of degree $p^2$ have $p+1$ decompositions over a sufficiently large field (see \ref{ex:add}). A classification of the polynomials of degree $p^2$ with at least two normal decompositions is given in the next theorem.

\begin{theorem}
\label{thm:gen}
Let $\ff$ be a field of characteristic $p>0$. Let $f$ be a normal polynomial in $\ff[x]$ of degree $p^2$ with $f' \neq 0$ and at least two normal decompositions. 
Then exactly one of the following statements holds:
\begin{enumerate}
\item\label{thm:gen:conj} There is $w \in \ff$ and a divisor $m$ of $p-1$ such that for each normal decomposition $(g,h)$ of $(x - f(w)) \circ f \circ (x + w)$ there are $a$ and $b$ in $\ff^\times$ such that $g = x(x^\ell - a)^m$ and $h= x(x^\ell - b)^m$ with $\ell=(p-1)/m$.

\item\label{thm:gen:neu} There is $w \in \ff$ and an integer $1 < m < p-1$ such that for each normal decomposition $(g,h)$ of $(x - f(w)) \circ f \circ (x + w)$ there is $r \in \{m, p - m\}$, and  $a, b \in \ff^\times$ such that $g = x^r (x - a)^{p-r}$, $h = x^{p-r} q$, and $h - a = (x - b)^r\tilde{q}$, where $q$ and $\tilde{q}$ are squarefree polynomials of degree $r$ and $p-r$, respectively.
\end{enumerate}
\end{theorem}

We note that $m$ depends only on $f$. Since $w$ is a root of $f(x) - f(w)$, the polynomial $(x - f(w)) \circ f \circ (x + w)$ is a conjugate of $f$. We will see that $w$ is unique in Case \ref{thm:gen:conj} and there are two alternative values for $w$ in Case \ref{thm:gen:neu}. Additive polynomials fall into Case \ref{thm:gen:conj} with $m=1$. The proof of this theorem will take the rest of this section.

Let  $(g_1, h_1)$ and $(g_2, h_2)$ be two normal decompositions of $f$. Then there are two intermediate fields $M$ and $N$ of $K(\alpha) \mid K(t)$ that correspond to $(g_1, h_1)$ and $(g_2, h_2)$, respectively. Throughout this section let $\prip$, $\priq$, and $\priP$ denote places in $M$, $N$, and $K(\alpha)$, respectively. We have that $M = K(h_1(\alpha))$ and $g_1 - t$ is the minimal polynomial of  $h_1(\alpha)$ over $\ff(t)$. Thus $\sum_{\prip \textrm{ finite}} d(\prip \mid P) = \deg(g_1')$, by \ref{lem:finiteRam}. Since $h_1 - h_1(\alpha)$ is the minimal polynomial of $\alpha$ over $M$ we have $\sum_{\priP \textrm{ finite}} d(\priP \mid \prip) = \deg(h_1')$. The analog holds for $N$. \ref{fig:fields} illustrates the relation between this fields.
\begin{figure}[h!]
\psset{unit=1.5mm}
\begin{pspicture}(-30,0)(30,30)
\rput(15,30){\rnode{Ka}{$K(\alpha)$}}
\rput(2,15){\rnode{M}{$M$}}
\rput(28,15){\rnode{N}{$N$}}
\rput(15,0){\rnode{K}{$K(t)$}}
\ncline[nodesep=3pt]{-}{K}{M}\naput{$g_1 - t$}
\ncline[nodesep=3pt]{-}{M}{Ka}\naput{$h_1 - h_1(\alpha)$}
\ncline[nodesep=3pt]{-}{K}{N}\nbput{$g_2 - t$}
\ncline[nodesep=3pt]{-}{N}{Ka}\nbput{$h_2 - h_2(\alpha)$}
\end{pspicture}
\caption{} \label{fig:fields}
\end{figure}

First we will show that we are in the situation in which we can apply \ref{lem:Dorey}. Then \ref{cor:dorey2}, \ref{lem:oneram} (which are similar to results in \cite{dorwha74}) and \ref{lem_zan} (which is due to \cite{zan93}) will tell us more about the ramification indices in $M$ and $N$. From this we can make a case distinction, whether there is an unramified place in $N$ over a certain place in $K(t)$ or not. This will lead to the two cases in the theorem.

\begin{claim}
$M N = K(\alpha)$.
\end{claim}
Clearly $M \subseteq M N \subseteq K(\alpha)$. If $M N = M$ then $N \subseteq M$, which can not be since $h_1 \neq h_2$. But since $[K(\alpha) \colon M N] \mid [K(\alpha) \colon M] = p$ we have $[K(\alpha) \colon M N] = 1$.
\begin{claim}
There is no finite place in $K(t)$ that is wildly ramified in $K(\alpha)$.
\end{claim}
Assume for contradiction that $P$ is wildly ramified, that is $p \mid e(\priP \mid P) = e(\priP \mid \prip) e(\prip \mid P)$ for $ \prip = \priP \cap M$. But then $p \mid e(\priP \mid \prip)$ or $p \mid e(\prip \mid P)$. Therefore we have $h_1 - b = (X - a)^p$ or $g_1 - b = (X- a)^p$, which is a contradiction to the assumption $f' \neq 0$.

Now we can apply \ref{lem:Dorey}. As in \cite{dorwha74} we need the notion of extra places. 

\begin{definition}
Define $$i(P, N \mid K(t)) = \sum_{\priq \mid P} d(\priq \mid P)$$ and $$i(P, K(\alpha) \mid M) = \sum_{\priP \mid P} d(\priP \mid \priP \cap M).$$
Call $P$ \emph{extra} in $N$ if $i(P, N\mid K(t)) > i(P, K(\alpha) \mid M)$.
\end{definition}

By Proposition 6.7 in \cite{gatgie10} we have $\deg_2(h_1) = \deg_2(g_2)$, where $\deg_2$ denotes the second degree (that is $\deg(f - ax^{\deg(f)})$ for a polynomial $f$ with leading coefficient $a$).
Since the degree of $h_1$ and $g_2$ is $p$ we have that the second degree is the degree of the derivative plus one and thus $\deg(h_1') = \deg(g_2')$. Then we get $d(\priqinf \mid \Pinf) = 2p - 2 - \deg(g_2') =  2p - 2 - \deg(h_1') = d(\priPinf \mid \pripinf) $, which proves that  $\Pinf$ cannot be extra in $N$.

Let $P$ be a finite place in $K(t)$ and let $\prip$ and $\priq$ be places over $P$ in $M$ and in $N$, respectively. Set $e=e(\prip \mid P)$ and $\et=e(\priq \mid P)$. For a place $\priP$ over $\prip$ and $\priq$ we have $e(\priP \mid P) = e(\priP \mid \prip) \cdot e$ and $e(\priP \mid \prip) = e(\priP \mid P) / e = \lcm(e, \et ) / e$. Thus $\sum_{\priP \mid \priq} d(\priP \mid \prip) = \sum ( \lcm(e, \et) / e - 1)= \gcd(e, \et) \cdot (\lcm(e, \et)/e - 1) = \et - \gcd(e, \et)$. We define $$c(\prip, \priq) = \sum_{\priP \mid \priq} d(\priP \mid \prip) = \et - \gcd(e, \et)$$ and note that $i(P, K(\alpha) \mid M) = \sum_{\prip, \priq} c(\prip, \priq)$. 

\begin{lemma}
\label{lem:c}
Let $P$ be a finite place in $K(t)$. Then $\sum_{\prip \mid P} c(\prip, \priq) \geq e(\priq \mid P) - 1$ for all places $\priq$ over $P$.
\end{lemma}
\begin{proof}
Let $P =\prod_{i=0}^l \prip_i^{e_i}$ in $M$. For $d= \gcd(e_i \mid 0 \leq i \leq l)$ we have $d \mid \sum e_i = p$. If $d > 1$ then we would have $d=p$ and $P$ would be wildly ramified in $M$, which cannot be. Thus $d=1$.

Let $\priq$ be a place over $P$ with ramification index $\et = e(\priq \mid P)$. Then as above we have $c(\prip_i , \priq) =  \et - \gcd(e_i, \et)$. If $\et = 1$ we have $\sum_i c(\prip_i, \priq) = \sum_i (\et - \gcd(e_i, \et))  = 0 =  \et - 1$. Thus assume $\et > 1$. Then $\et$ cannot divide $e_i$ for all $i$, since their gcd is one. We distinguish two cases:

Case 1: $\et$ divides all but one places $\prip$ over $P$ in $M$. Then let $\prip_0$ be the place such that $\et \nmid e_0$. The gcd of $\et$ and $e_0$ divides $\et$ and thus divides all ramification indies of places over $P$ in $M$. But their gcd is one. Thus the gcd of $\et$ and $e_0$ is one and we have $\sum_{\prip \mid P} c(\prip, \priq) \geq c(\prip_0, \priq) = \et - 1$.

Case 2: There are at least two places, say  $\prip_1$ and $\prip_2$, over $P$ in $M$ such that $\et \nmid e_i$ for $i=1,2$. Then we have $\et / \gcd(e_i, \et) > 1 $ since else we would have  $\et \mid e_i$, which is a contradiction. Thus $\gcd(e_i, \et) \leq \et / 2$. Hence $\et - \gcd(e_i, \et) \geq \et/ 2$ and thus $\sum_i c(\prip_i , \priq) \geq c(\prip_1, \priq) + c(\prip_2, \priq) \geq \et > \et - 1$, as claimed. 
\end{proof}

\begin{corollary}
\label{cor:dorey2}
There is no finite place in $K(t)$ which is extra in $N$.
\end{corollary}
\begin{proof}
Let $P$ be a finite place. By \ref{lem:c} we have   $\sum_{\prip \mid P} c(\prip, \priq) \geq e(\priq \mid P) - 1$ for all $\priq$. But then $i(P, K(\alpha) \mid M) = \sum_{\prip, \priq} c(\prip, \priq) \geq \sum_\priq (e(\priq \mid P) - 1) = i(P, N\mid K(t))$ which shows that $P$ is not extra in $N$.
\end{proof}

As seen in \ref{sec_prel} we have by the Hurwitz Genus Formula $$\sum_P i(P, N \mid K(t) ) =  \sum_{\priq \mid P} d(\priq \mid P) = 2p - 2 = \sum_P i(P, K(\alpha) \mid M).$$ But since there are no extra places in $N$ we get $i(P, N \mid K(t) )= i(P, K(\alpha) \mid M)$ for all places $P$. 

\begin{lemma}
\label{lem_zan}
Let $P$ be a finite place in $K(t)$. Then the following statements hold:
\begin{enumerate}
\item For each ramified place $\priq$ over $P$ in $N$ the ramification index $e(\priq \mid P)$ divides $e(\prip \mid P)$ for all but one place $\prip$ over $P$ in $M$. 
\item For each ramified place $\prip$ over $P$ in $M$ the ramification index $e(\prip \mid P)$ divides $e(\priq \mid P)$ for all but one place $\priq$ over $P$ in $N$. 
\item $P$ is ramified in $M$ if and only if it is ramified in $N$.
\end{enumerate}
\end{lemma}
\begin{proof}
To prove  (i), we claim that the second case in the proof of \ref{lem:c} does not occur. For $\priq$ falling into Case 2 we had seen that $\sum_{\prip \mid P} c(\prip, \priq) \geq e(\priq \mid P)$. Set $\varepsilon(\priq) = 1$ if $\priq$ falls into this case and $\varepsilon(\priq) = 0$ else. Then we have in any case $\sum_{\prip \mid P} c(\prip, \priq) \geq e(\priq \mid P) - 1 + \varepsilon(\priq)$. Hence we get $i(P, K(\alpha) \mid M) = \sum_{\prip, \priq} c(\prip, \priq) \geq \sum_{\priq} e(\priq \mid P) - 1 + \varepsilon(\priq) = i(P, N \mid K(t) ) + \sum_\priq \varepsilon(\priq)$. But since $i(P, N \mid K(t) )= i(P, K(\alpha) \mid M)$ we have $\sum_\priq \varepsilon(\priq) = 0$. This proves the claim.

The second statement can be proven analogously to the first one, by interchanging the role of $M$ and $N$ in the previous results.

Finally, if $P$ is ramified in $N$ then by (i) there is a place $\priq$ with $1 < e(\priq \mid P) \mid e(\prip \mid P)$ for some place $\prip$ in $M$. Thus $P$ is ramified in $M$. The other direction follows in the same way from (ii).
\end{proof}

\begin{lemma} \label{lem:oneram}
There is at most one finite place in $K(t)$ that is ramified in $M$. Moreover if there is a place that is ramified in $M$ then it has at most one unramified factor.
\end{lemma}
\begin{proof}
Let $P$ be a finite place in $K(t)$, which is ramified in $M$. Assume there is a place $\prip$ such that  $e(\prip \mid P ) = 1 $. Then $\sum _{\priP \mid \prip} d(\priP \mid \prip) = \sum_\priq c(\prip, \priq) \geq  \sum_\priq e(\priq \mid P) - \gcd(1, e(\priq \mid P)) = \sum_\priq (e(\priq \mid P) - 1) = i(P, N \mid K(t) )$. If there are two unramified places $\prip_1$ and $\prip_2$ then $i(P, K(\alpha) \mid M) \geq \sum _{\priP \mid \prip_1} d(\priP \mid \prip_1)  + \sum _{\priP \mid \prip_2} d(\priP \mid \prip_2)  \geq 2 i(P, N \mid K(t) )$. But since $i(P, K(\alpha) \mid M) = i(P, N \mid K(t) )$ this can only be if $i(P, N \mid K(t) )=0$. Hence $P$ is unramified in $N$ and by \ref{lem_zan} unramified in $M$, in contradiction to our assumption. Thus there can be at most one unramified place over $P$. If $\rho$ denotes the number of places in $M$ over $P$ we have $1+ 2(\rho-1) \leq \sum e(\prip \mid P) = p$ and thus $\rho \leq (p+1)/2$. Therefore $i(P, M \mid K(t) ) = p - \rho \geq (p-1)/2$. But since $\sum_{P \textrm{ finite}} i(P, M \mid K(t)) = \deg(g_1') < p-1$ there can be at most one such a place in $K(t)$.
\end{proof}

Setting $k = \deg_2(g_1)$, we have $k = \deg_2(g_i) = \deg_2(h_i)$ for $i = 1, 2$ (see Proposition 6.7 in \cite{gatgie10}). We have a special case when $k = 1$. Then $\sum_{P \textrm{ finite}} i(P , M \mid K(t)) = \deg(g_1') = k -1=0$ and thus there are no ramified places. But since the second degree of the decompositions equals one, we have in this case that $g_1$ and $h_1$ are of the form $x^p - ax$ and $x^p - bx$, respectively. Thus $f$ is an additive polynomial and we are in the Case (\bare\ref{thm:gen:conj}) of \ref{thm:gen}.

We assume $ k > 1$. Then we have $\sum_{P \textrm{ finite}} i(P , M \mid K(t)) = k -1 > 0$ and thus there is a finite place $P$ which is ramified in $M$. By \ref{lem:oneram} we have $P$ is the only finite place which is ramified in $M$ and thus $i(P , M \mid K(t)) = k -1$. By \ref{lem_zan} and by interchanging $M$ and $N$ in \ref{lem:oneram} we get that $P$ is the only place that is ramified in $N$. Then we have $i(P , M \mid K(t)) =  p - \rho_1 = k-1 = p - \rho_2 = i(P , N \mid K(t))$, where $\rho_1$ and $\rho_2$ are the numbers of places over $P$ in $M$ and $N$, respectively. Thus $\rho_1 = \rho_2 = p - k + 1=\ell+1$ with $\ell = p-k$. 

By the correspondence between ramification and multiplicities we get that there is exactly one $c$ in $K$ such that $f-c$ has multiple roots. Then for each automorphism $\sigma$ of $K$ that fixes $\ff$ we have $f - \sigma (c) = \sigma(f - c)$ has multiple roots and thus $\sigma(c) = c$. This proves that $c$ is in $\ff$. We will later see that there is also a root $w \in \ff$ of $f- c$, that is $f(w) = c$. Thus the  conjugate $(x - f(w)) \circ f \circ (x + w)$ lies in $\ff[x]$ and has multiple roots. 

Now assume that there is an unramified place $\priq_0$ over $P$ in $N$. We will prove that $f$ falls into Case (\bare\ref{thm:gen:conj}) of \ref{thm:gen}. 

Let $\priq_i$ be the ramified places over $P$ in $N$ with ramification indices $\et_i$, for $0 \leq i \leq \ell$. Then let $e_0$ denote the ramification index in $M$ that is not divided by $\et_1$. If $e_0 \neq 1$ then $e_0$ cannot divide $\et_0 = 1$ and thus it must divide $\et_1$. But this would imply that $e_0$ divides all ramification indices in $M$, which cannot be. Thus we get $e_0 = 1$. Then $\et_i$ divides all ramification indices in $M$ that are greater than 1, and the other way round. Thus all of this ramification indices equal, say we have $m = e_i = \et_i$ for all $1 \leq i \leq \ell$. 

Then we get that $g_1 - c$ is of the form $(x - \tilde{a})\tilde{g}^m$, where $\tilde{a}$ is in $K$ and $\tilde{g}$ is a squarefree polynomial of degree $\ell$. We have that $g_1 - c$ is a polynomial over $\ff$ and the irreducible factors of $g_1 -c$ over $\ff$ have only simple roots. Thus $\tilde{g}$ is defined over $\ff$ and $\tilde{a}$ is in $\ff$.

Since $e(\priP \mid P) = \lcm(e(\prip \mid P), e(\priq \mid P) ) = m$ or $=1$ we get $e(\priP \mid \prip) = m$ if and only if $e(\prip \mid P)=1$ and $e(\priq \mid P) = m$. Thus only $\prip_0$ is ramified in $K(\alpha) \mid M$ and has the same ramification like $P$ in $N \mid K(t)$. Thus as above $h_1 - \tilde{a}$ is of the form $(x - w)\tilde{h}^m$, where $w \in \ff$ and $\tilde{h}$ is a squarefree polynomial of degree $\ell$. Now we conjugate as follows: $(x - c) \circ f \circ (x + w) = (x - c) \circ g_1 \circ (x + \tilde{a}) \circ (x - \tilde{a}) \circ h_1 \circ (x + w) = x\hat{g}^m \circ x\hat{h}^m$, with $\hat{g} = \tilde{g} \circ (x + \tilde{a})$ and $\hat{h} = \tilde{h} \circ (x + w)$.

To see that we are in Case (\bare\ref{thm:gen:conj}) of \ref{thm:gen} we have to prove that $\hat{g}$ and $\hat{h}$ are of the appropriate form. For a polynomial $g = x\hat{g}^m$ the derivative of $g$ is $g' = \hat{g}^{m-1} (\hat{g} + m x\hat{g}' )$. If $a_i$ are the coefficients of $\hat{g}$ then we have $\hat{g} + m x\hat{g}'  = \sum_i (a_i + m i a_i)x^i$. Let $\ell'$ be the degree of  $\hat{g} + m x\hat{g}' $. Then $k-1 = \deg(g') = (m-1)\ell + \ell' = m\ell - \ell + \ell' = p - \ell - 1 + \ell' = k-1 + \ell'$ and thus $\ell'=0$. But this means that we have $a_i + m i a_i = 0$ for all $1 \leq i \leq \ell$. For $ 1 \leq i < \ell$ this is the case only if $a_i = 0$. Thus we get $\hat{g} = (x^\ell - a)$ and $g$ is of the form as claimed.

Now we consider the other case, where there is no unramified place over $P$ in $N$. Then there is neither one in $M$. To prove that we are actually in Case (\bare\ref{thm:gen:neu}), we first prove that $\ell=1$. For this propose we translate \ref{lem_zan} into the language of graphs. Let  $V = A \cup B$ be the set of vertices, where $A = \{e_i \colon 0\leq i\leq \ell\}$ and $B = \{\et_i \colon 0\leq i\leq \ell\}$. Let the set of edges $E$ consist of $(e_i, \et_j)$ if $e_i | \et_j$ and of $(\et_i, e_j)$ if $\et_i | e_j$. Then this yields a directed bipartite graph with outdegree $\delta (v) = \ell$ for each $v \in V$. Note that if there is a vertex in $A$ which is connected to all other vertices in $A$ then we get that the gcd of all $e_i$ is greater then one, which is a contradiction.

\begin{lemma}
Let $G=(V ,E)$ be a directed bipartite graph, with bipartition $V = A \cup B$. Assume $A$ and $B$ have the same cardinality $\ell+1 > 2$ and the outdegree of each vertex equals $\ell$. Then there is $a \in A$ such that $a$ is connected to all vertices in $A$ or there is $b \in B$ such that $b$ is connected to all vertices in $B$.
\end{lemma}
\begin{proof}
Assume that there are $a_0$ and $a_1$ in $A$ such that there is no $a_1$-$a_0$ path (if such $a_0$ and $a_1$ would not exist we would be done). Since the outdegree of $a_1$ is $\ell$ there is $b_0$ in $B$ such that $(a_1, b_0) \not\in E$. Then for all $b$ in $B' = B \setminus \{b_0\}$ there is a edge $(a_1, b)$ and thus no edge $(b, a_0)$. If now $(b_0, a_1)$ is in $G$ then $b_0$ is connected to all $b$ in $B'$ via $a_1$ and we would be done. Thus we assume that $(b_0, a_1)$ is not in $G$. 

We claim that $G' = G \setminus \{a_0, b_0\}$ is a complete bipartite graph. We have already seen that for each $b \in B'$ the edge $(b, a_0)$ is not in $G$. Thus $b$ has outdegree $\ell$ in $G'$. Let $a$ be in $A' = A\setminus \{a_0\}$. Then $a$ has outdegree $\ell$ in $G$ and it would have less outdegree in $G'$ only if there would be $(a, b_0)$ in $G$. Then for $b$ in $B'$ we get that $(a_1, b), (b, a), (a, b_0), (b_0, a_0)$ is a $a_1$-$a_0$ path, which is a contradiction to our assumption. Thus $G'$ is complete. 

Since the outdegree of $a_0$ is $\ell > 1$ there is a $b_1$ in $B'$ such that $(a_0, b_1)$ is in $G$. Let $b$ be any vertex in $B'$. Since $G'$ is complete there is a $b_1$-$b$ path $p$. But then $(b_0, a_0), (a_0, b_1), p$ is a $b_0$-$b$ path and thus $b_0$ is connected to all vertices in $B$.
\end{proof}

By the previous discussion we know that $\ell = 1$ and therefore $P$ splits exactly into two ramified places in both fields, $M$ and $N$, say  $P = \prip_0^m \prip_1^{p-m}$ in $M$ and $P = \priq_0^m\priq_1^{p-m}$ in $N$, with $1 < m < p-1$. But this means that $g_1 - c$ is of the form $(x - a)^m (x- \tilde{a})^{p - m}$, for suitable $a, \tilde{a} \in \ff$.

Now there are $\gcd(m,m)=m$ places $\priP$ over $\prip_0$ and $\priq_0$. For such a place we have $e(\priP \mid \prip_0) = \lcm (m,m) / m = 1$. Furthermore there is one place $\priP$ over $\prip_0$ and $\priq_1$ with $e(\priP \mid \prip_0) = \lcm (m,p-m) / m = p-m$. Thus $h_1-a$ is of the form $(x-b)^{p-m} q$, where $q$ is a squarefree polynomial of degree $m$. Similarly  we get $h_1 - \tilde{a} = (x - \tilde{b})^m\tilde{q}$. It is left to prove that $a$, $\tilde{a}$, $b$, and $\tilde{b}$ are in $\ff$. Assume $a$ and $\tilde{a}$ would not be in $\ff$. Then $(x - a)(x-\tilde{a})$ must be an irreducible factor of $g_1 - c$. But then $m = p-m$ which is a contradiction (note that for $p = 2$ there is anyway no such $m$). Similarly we get that $b$ and $\tilde{b}$ are in $\ff$. By conjugating $f$ with $x + w$ for $w \in \{b, \tilde{b} \}$, we achieve the desired form.

Finally we note that the case, in which the polynomial falls, depends on whether there is an unramified place over $P$ or not. Thus the two cases are distinct. This finishes the proof of \ref{thm:gen}.

\subsection{Parametrization}
\label{subsec:conj}
The polynomials examined in Section 6 in \cite{gatgie10} fall into Case (\bare\ref{thm:gen:conj}) of \ref{thm:gen}.
\begin{theorem}
\label{ex:conj}
For parameters $\varepsilon \in \{ 0, 1 \}$, $u$, $s \in \mathbb{F}^\times$ and $\ell$ a positive divisor of $p - 1$ let $f$ be the polynomial 
$$f=x(x^{\ell(p+1)} - \varepsilon u s^r x^\ell + u s^{p+1})^m,$$
with $m = (p-1)/\ell$. 
Then $f$ has for each root $t$ of $x^{p+1} - \varepsilon u x + u$ in $\ff$ a decomposition $(g, h)$ with $g = x(x^\ell - u s^p t^{-1})^m$ and $h= x (x^\ell - s t)^m$. These decompositions are pairwise distinct and there are no other possible decompositions of $f$.
\end{theorem}
\begin{proof}
See Theorem 6.2 in \cite{gatgie10}.
\end{proof}

\begin{corollary}
All polynomials which fall into Case (\bare\ref{thm:gen:conj}) of \ref{thm:gen} can be parametrized as in \ref{ex:conj}.
\end{corollary}
\begin{proof}
We have $g = x(x^\ell - a)^m$ and $h = x(x^\ell - b)^m$ for suitable $a$ and $b$ in $\ff^\times$. Then we get $f = x(x^\ell - b)^m((x(x^\ell-b)^m)^l - a)^m = x(x^{\ell(p+1)} - (b^p +a)x^\ell +ab)^m$. Now define $\varepsilon = 0$ if $b^p +a = 0$ else define $\varepsilon = 1$. In case $\varepsilon = 0$ we set $s=1$, $t=b$, and $u= ab$. If $\varepsilon =1$ we set $s= ab (b^p + a)^{-1}$, $t= b/s$, and $u = ab/ s^{p+1}$. In both cases we have that $u$, $s$, and $t$ are in $\ff$ and the equations $t^{p+1} -\varepsilon u t + u =0$, $b=st$, and $a = us^p t^{-1}$ hold as claimed.
\end{proof}

Note that if the field $\ff$ is large enough, the polynomial $x^{p+1} - \varepsilon u x + u$ has $p+1$ roots and thus $f$ has $p+1$ (minimal) decompositions. This is another example, that shows that the bounds in \ref{subsec:ub} are sharp. Still open is a closer examination of the polynomials falling into Case (\bare\ref{thm:gen:neu}).

\newpage

\bibliography{journals,refs,lncs,DAveroe_local}

\providecommand{\ymd}[3]{\csname @ifempty\endcsname{#1}{?#1/#2/#3?}{\csname
  @ifempty\endcsname{#2}{?#1/#2/#3?}{\csname
  @ifempty\endcsname{#3}{?#1/#2/#3?}{{\relax \day=#3\relax \month=#2\relax
  \year=#1\relax \number\day~\ifcase\month\or January\or February\or March\or
  April\or May\or June\or July\or August\or September\or October\or November\or
  December\fi \space\ifnum\year>0\relax \number\year \else \csname
  count@\endcsname1\relax \expandafter\advance\csname count@\endcsname-\year
  \expandafter\number\csname count@\endcsname~BC\fi}}}}}
  \providecommand{\hide}[1]{.} \providecommand{\Hide}[1]{\unskip}
  \providecommand{\gobble}[1]{} \providecommand{\todo}[1]{\textbf{`?`?`?#1???}}
  \providecommand{\Name}[1]{#1} \providecommand{\Textgreek}[1]{\textgreek{#1}}
  \providecommand{\at}{\char64\relax}
  \providecommand{\cyr}{\PackageError{cyrillic}{Package not loaded. Use
  \string\usepackage{cyrillic} to define \string\cyr\space appropriately}{}
  \gdef\cyr{\def\cprime{c'}{\bf ?cyr?}}\cyr}
  \makeatletter\protected@write\@auxout{}{\string
  \gdef\string\abbr{\string\csname\space
  @gobble\string\endcsname}}\gdef\abbr{}\makeatother
\begin{thebibliography}{18}
\providecommand{\natexlab}[1]{#1}
\providecommand{\url}[1]{\texttt{#1}}
\providecommand{\urlprefix}{URL }
\providecommand{\selectlanguage}[1]{\relax}

\bibitem[{Atkinson(1975)}]{atk75}
\textsc{M.~D. Atkinson} (1975).
\newblock An algorithm for finding the blocks of a permutation group.
\newblock \emph{Mathematics of Computation} \textbf{29}, 911--913.
\newblock ISSN 0378-4754.

\bibitem[{Bostan \emph{et~al.}(2004)Bostan, Lecerf, Salvy, Schost \&
  Wiebelt}]{boslec04}
\textsc{A.~Bostan}, \textsc{G.~Lecerf}, \textsc{B.~Salvy}, \textsc{{{\'E}}.
  Schost} \& \textsc{B.~Wiebelt} (2004).
\newblock Complexity issues in bivariate polynomial factorization.
\newblock In \emph{ISSAC 2004: Proceedings of the 2004 International Symposium
  on Symbolic and Algebraic Computation}, \textsc{Rudolf Fleischer} \&
  \textsc{Gerhard Trippen}, editors, 42--49. ACM.
\newblock \urlprefix\url{http://dx.doi.org/10.1145/1005285.1005294}.

\bibitem[{Butler(1992)}]{but92}
\textsc{Greg Butler} (1992).
\newblock An analysis of Atkinson's algorithm.
\newblock \emph{SIGSAM Bull.} \textbf{26}, 1--9.
\newblock ISSN 0163-5824.
\newblock \urlprefix\url{http://doi.acm.org/10.1145/130933.130935}.

\bibitem[{Dorey \& Whaples(1974)}]{dorwha74}
\textsc{F.~Dorey} \& \textsc{G.~Whaples} (1974).
\newblock Prime and Composite Polynomials.
\newblock \emph{Journal of Algebra} \textbf{28}, 88--101.
\newblock \urlprefix\url{http://dx.doi.org/10.1016/0021-8693(74)90023-4}.

\bibitem[{Fried \& MacRae(1969)}]{frimac69}
\textsc{Michael~D. Fried} \& \textsc{R.~E. MacRae} (1969).
\newblock On the invariance of chains of Fields.
\newblock \emph{Illinois Journal of Mathematics} \textbf{13}, 165--171.

\bibitem[{von~zur Gathen(1990{\natexlab{a}})}]{gat90c}
\textsc{Joachim von~zur Gathen} (1990{\natexlab{a}}).
\newblock Functional Decomposition of Polynomials: the Tame Case.
\newblock \emph{Journal of Symbolic Computation} \textbf{9}, 281--299.
\newblock \urlprefix\url{http://dx.doi.org/10.1016/S0747-7171(08)80014-4}.

\bibitem[{von~zur Gathen(1990{\natexlab{b}})}]{gat90d}
\textsc{Joachim von~zur Gathen} (1990{\natexlab{b}}).
\newblock Functional Decomposition of Polynomials: the Wild Case.
\newblock \emph{Journal of Symbolic Computation} \textbf{10}, 437--452.
\newblock \urlprefix\url{http://dx.doi.org/10.1016/S0747-7171(08)80054-5}.

\bibitem[{von~zur Gathen(2009)}]{gat09b}
\textsc{Joachim von~zur Gathen} (2009).
\newblock The Number of Decomposable Univariate Polynomials --- Extended
  Abstract.
\newblock 359--366.
\newblock ISBN 978-1-60558-609-0.
\newblock Preprint (2008) at {\tt{http://arxiv.org/abs/0901.0054}}.

\bibitem[{von~zur Gathen \& Gerhard(1999)}]{gatger99}
\textsc{Joachim von~zur Gathen} \& \textsc{J{\"{u}}rgen Gerhard} (1999).
\newblock \emph{Modern Computer Algebra}.
\newblock Cambridge University Press, Cambridge,~UK, {First} edition.
\newblock ISBN 0-521-64176-4.
\newblock \urlprefix\url{http://cosec.bit.uni-bonn.de/science/mca/}.
\newblock Other available editions: Second edition 2003, Chinese edition,
  Japanese translation.

\bibitem[{von~zur Gathen \emph{et~al.}(2010)von~zur Gathen, Giesbrecht \&
  Ziegler}]{gatgie10}
\textsc{Joachim von~zur Gathen}, \textsc{Mark Giesbrecht} \& \textsc{Konstantin
  Ziegler} (2010).
\newblock Composition collisions and projective polynomials. Statement of
  results.
\newblock Preprint available at \url{http://arxiv.org/abs/1005.1087}.

\bibitem[{Landau \& Miller(1985)}]{lanmil85}
\textsc{S.~Landau} \& \textsc{G.~L. Miller} (1985).
\newblock Solvability by Radicals is in Polynomial Time.
\newblock \emph{Journal of Computer and System Sciences} \textbf{30}, 179--208.

\bibitem[{Lecerf(2008)}]{lec08}
\textsc{G.~Lecerf} (2008).
\newblock Fast Separable Factorization and Applications.
\newblock \emph{Applicable Algebra in Engineering, Communication and Computing}
  \textbf{19}(2), 135--160.
\newblock
  \urlprefix\url{http://www.springerlink.com/content/75430606n6r52k67/fulltext%
.pdf}.

\bibitem[{Neukirch(2007)}]{neu99}
\textsc{J{\"u}rgen Neukirch} (2007).
\newblock \emph{Algebraische Zahlentheorie}.
\newblock Springer-Verlag, Berlin.
\newblock ISBN 3-540-37547-3.

\bibitem[{Ritt(1922)}]{rit22}
\textsc{J.~F. Ritt} (1922).
\newblock Prime and Composite Polynomials.
\newblock \emph{Transactions of the American Mathematical Society} \textbf{23},
  51--66.
\newblock \urlprefix\url{http://www.jstor.org/stable/1988911}.

\bibitem[{Stichtenoth(1993)}]{sti93}
\textsc{Henning Stichtenoth} (1993).
\newblock \emph{Algebraic Function Fields and Codes}.
\newblock Springer-Verlag, 260 pages .

\bibitem[{Wielandt(1964)}]{wie64}
\textsc{Helmut Wielandt} (1964).
\newblock \emph{Finite permutation groups}.
\newblock Academic Press, New York.
\newblock ISBN 0-127-49656-4.
\newblock Translated from the German by R. Bercov.

\bibitem[{Zannier(1993)}]{zan93}
\textsc{U{\hide{mberto}}~Zannier} (1993).
\newblock Ritt's Second Theorem in arbitrary characteristic.
\newblock \emph{Journal f{\"{u}}r die reine und angewandte Mathematik}
  \textbf{445}, 175--203.
\newblock
  \urlprefix\url{http://www.digizeitschriften.de/index.php?id=loader&tx_jkDigi%
Tools_pi1[IDDOC]=503382}.

\bibitem[{Zippel(1991)}]{zip91}
\textsc{Richard Zippel} (1991).
\newblock Rational Function Decomposition.
\newblock In \emph{Proceedings of the 1991 International Symposium on Symbolic
  and Algebraic Computation ISSAC~'91, {\rm Bonn, Germany}}, \textsc{Stephen~M.
  Watt}, editor, 1--6. ACM Press, Bonn, Germany.
\newblock ISBN 0-89791-437-6.

\end{thebibliography}

\end{document}